\documentclass[11pt]{amsart}

\usepackage[colorlinks=true,urlcolor=blue, citecolor=red,linkcolor=blue,
linktocpage,pdfpagelabels, bookmarksnumbered,bookmarksopen]{hyperref}
\usepackage[hyperpageref]{backref}
\usepackage{geometry}
\usepackage{amssymb}
\usepackage{graphicx}
\usepackage{verbatim}

\numberwithin{equation}{section}

\newtheorem{teo}{Theorem}[section]
\newtheorem{lm}[teo]{Lemma}
\newtheorem{prop}[teo]{Proposition}
\newtheorem{exa}[teo]{Example}
\theoremstyle{definition}
\newtheorem{defi}[teo]{Definition}
\newtheorem{oss}[teo]{Remark}
\newtheorem{open}{Open problem}
\newtheorem*{ack}{Acknowledgments}

\title[Constrained critical points of Dirichlet integrals]{An overview on constrained critical points\\ of Dirichlet integrals}

\author[Brasco]{Lorenzo Brasco}
\address[L.\ Brasco]{Dipartimento di Matematica e Informatica
\newline\indent
Universit\`a degli Studi di Ferrara
\newline\indent
Via Machiavelli 35, 44121 Ferrara, Italy}
\email{lorenzo.brasco@unife.it}

\author[Franzina]{Giovanni Franzina}
\address[G. Franzina]{Istituto Nazionale di Alta Matematica (INdAM)
\newline\indent
Unit\`a di Ricerca di Firenze c/o DiMaI ``Ulisse Dini'' 
\newline\indent 
Universit\`a degli Studi di Firenze
\newline\indent 
Viale Morgagni 67/A, 50134 Firenze, Italy}
\email{franzina@math.unifi.it}


\begin{document}

\begin{abstract}
We consider a natural generalization of the eigenvalue problem for the Laplacian with homogeneous Dirichlet boundary conditions. This corresponds to look for the critical values of the Dirichlet integral, constrained to the unit $L^q$ sphere.
We collect some results, present some counter-examples and compile a list of open problems.
\end{abstract}

\maketitle

\begin{center}
\begin{minipage}{8cm}
\small
\tableofcontents
\end{minipage}
\end{center}

\section{Introduction}

\subsection{The spectrum of the Laplacian}
Let us consider an open and bounded set $\Omega\subset\mathbb{R}^N$, with $N\ge 2$. By means of the {\it Spectral Theorem} for positive, compact and self-adjoint operators (see for example \cite[Theorem 1.2.1]{He}), it is a classical fact that the {\it Helmholtz equation} 
\begin{equation}
\label{autovalore}
-\Delta u=\lambda\,u,\ \mbox{ in } \Omega,\qquad u=0, \ \mbox{ on }\partial\Omega,
\end{equation}
admits nontrivial solutions only for a discrete set of values $\lambda>0$, called {\it eigenvalues of the Dirichlet-Laplacian on $\Omega$}. The corresponding nontrivial solution $u$ is called {\it eigenfunction} and the pair $(u,\lambda)$ is usually referred to as {\it eigenpair}.
\par
Here solutions are always intended in weak sense, i.e. the eigenfunctions $u$ must belong to the homogeneous Sobolev space $\mathcal{D}^{1,2}_0(\Omega)$. The latter is defined as the completion of $C^\infty_0(\Omega)$ with respect to the norm
\[
\varphi\mapsto \|\nabla \varphi\|_{L^2(\Omega)}=\left(\int_\Omega |\nabla \varphi|^2\,dx\right)^\frac{1}{2}.
\]
Observe that the space $\mathcal{D}^{1,2}_0(\Omega)$ can be defined for general open sets, not necessarily bounded. Moreover, when the set $\Omega$ supports the Poincar\'e inequality
\[
C_\Omega\,\int_\Omega |\varphi|^2\,dx\le \int_\Omega |\nabla \varphi|^2\,dx,\qquad \mbox{ for every } \varphi\in C^\infty_0(\Omega),
\]
we have that $\mathcal{D}^{1,2}_0(\Omega)$ coincides with the more familiar space $H^1_0(\Omega)$, defined as the closure of $C^\infty_0(\Omega)$ in the standard Sobolev space $H^1(\Omega)$.
\par
By the Lagrange's Multipliers Rule, it is easily shown that each eigenvalue can be seen as a critical value of the 
Dirichlet integral
\[
\varphi\mapsto \int_\Omega |\nabla \varphi|^2\,dx,
\]
constrained to the ``manifold''
\[
\mathcal{S}_2(\Omega)=\Big\{\varphi\in \mathcal{D}^{1,2}_0(\Omega)\, :\, \|\varphi\|_{L^2(\Omega)}=1\Big\}.
\]
The associated eigenfunctions are then the corresponding critical points.
\par 
By using the well-known spectral properties of the Dirichlet-Laplacian (see \cite[Chapter 1]{He}), we can single out the following remarkable properties of these critical values:
\begin{enumerate}
\item[$(\mathcal{E}_1)$]
\label{E1} the Dirichlet integral
\[
\int_\Omega |\nabla \varphi|^2\,dx,
\]
constrained to the unit sphere of $L^2(\Omega)$, only admits a discrete sequence of positive critical values, accumulating to $+\infty$. We indicate it with
\[
\mathrm{Spec}(\Omega)=\{\lambda_1(\Omega),\lambda_2(\Omega),\dots\};
\]
\vskip.2cm
\item[$(\mathcal{E}_2)$] the corresponding critical points give an orthonormal basis of $L^2(\Omega)$;
\vskip.2cm
\item[$(\mathcal{E}_3)$] the constrained problem admits a global minimum, which coincides with the first eigenvalue $\lambda_1(\Omega)$;
\vskip.2cm
\item[$(\mathcal{E}_4)$] if $\Omega$ is connected, then $\lambda_1(\Omega)$ is {\it simple}, i.e. global minimizers on $\mathcal{S}_2(\Omega)$ are unique, up to the choice of the sign. Morever, this is the only critical value with constant-sign eigenfunctions;
\vskip.2cm
\item[$(\mathcal{E}_5)$] if $\Omega$ has $\#$ connected components $\{\Omega_j\}_j$, then 
\[
\bigcup_{j=1}^\#\mathrm{Spec}(\Omega_j)= \mathrm{Spec}(\Omega),
\]
and
\begin{equation}
\label{primino}
\lambda_1(\Omega)=\min_{j} \lambda_1(\Omega_j);
\end{equation}
\vskip.2cm
\item[$(\mathcal{E}_6)$] each critical value has a variational characterization, given for example by the {\it Courant-Fischer-Weyl min-max principle}
\begin{equation}
\label{CFW}
\lambda_k(\Omega)=\min_{\mathcal{F}\subset \Sigma_k(\Omega)} \left\{\max_{\varphi\in \mathcal{F}}\int_\Omega |\nabla \varphi|^2\,dx\right\},\qquad k\in\mathbb{N}\setminus\{0\},
\end{equation}
where 
\[
\Sigma_k(\Omega)=\Big\{\mathcal{F}=F\cap\mathcal{S}_2(\Omega)\,:\, F \mbox{ $m$-dimensional subspace of } \mathcal{D}^{1,2}_0(\Omega)\, :\, m\ge k\Big\}.
\]
\end{enumerate}

\subsection{The $q-$spectrum of the Laplacian} We now try to revert the point of view and adopt directly the one of Critical Point Theory. Then we ask the following simple question: 
\vskip.2cm
{\it what can be said about the critical values of the Dirichlet integral, constrained to the unit sphere
\[
\mathcal{S}_{q}(\Omega)=\Big\{u\in \mathcal{D}^{1,2}_0(\Omega)\, :\, \|u\|_{L^q(\Omega)}=1\Big\},
\]
with $q\not =2$?}
\begin{oss}
In this paper, we always consider the case $q>1$ and $q<2^*$, where the latter is the critical Sobolev exponent. The cases $q=1$ and $q\ge 2^*$ are certainly interesting (in the first case, the notion of critical value should be carefully adapted), but they present additional difficulties and they will not be considered here.
\end{oss}
We point out that switching from $L^2$ to $L^q$ completely destroys the Hilbertian structure of the problem. Thus, we can not expect to obtain a linear eigenvalue--type equation, nor to apply the standard tools of Spectral Theory to answer the question above.
\par
More precisely, by the Lagrange's Multipliers Rule, we see that in this new setting the critical values $\lambda$ are those numbers for which the {\it Lane-Emden equation}
\begin{equation}
\label{autosalone}
-\Delta u=\lambda\, |u|^{q-2}\,u,\ \mbox{ in } \Omega,\qquad u=0, \ \mbox{ on }\partial\Omega,
\end{equation}
admits nontrivial solutions. We point out that equation \eqref{autosalone} has to be coupled with the normalization $u\in\mathcal{S}_q(\Omega)$. If one wants to get rid of this normalization, the correct version of this eigenvalue equation is\footnote{This corresponds to look at nontrivial critical points of the {\it Rayleigh--type quotient}
\[
\varphi\mapsto\frac{\displaystyle \int_\Omega |\nabla \varphi|^2\,dx}{\displaystyle\left(\int_\Omega |\varphi|^q\,dx\right)^\frac{2}{q}}.
\]} 
\begin{equation}
\label{autosalonebis}
-\Delta u=\lambda\,\|u\|_{L^q(\Omega)}^{2-q}\, |u|^{q-2}\,u,\ \mbox{ in } \Omega,\qquad u=0, \ \mbox{ on }\partial\Omega.
\end{equation}
Observe that the right-hand side is mildly nonlocal, due to the presence of the $L^q$ norm. 
\par
We can define the {\it $q-$spectrum of the Dirichlet-Laplacian on $\Omega$} as
\[
\mathrm{Spec}(\Omega;q)=\Big\{\lambda\in \mathbb{R}\, :\, \mbox{equation \eqref{autosalonebis} admits a solution in } \mathcal{D}^{1,2}_0(\Omega)\setminus\{0\}\Big\}.
\]
Accordingly, we call any element of this set a {\it $q-$eigenvalue of $\Omega$}. A corresponding solution $u$ will be called {\it $q-$eigenfunction} and the pair $(u,\lambda)$ will be referred to as {\it $q-$eigenpair}.
\begin{oss}[Unconstrained critical points] It is useful to keep in mind that the eigenvalue problem considered in this paper is equivalent to the problem of finding critical points of the ``free'' functional 
\[
\mathfrak{F}_q(\varphi)=\frac{1}{2}\,\int_\Omega |\nabla \varphi|^2\,dx-\frac{1}{q}\,\int_\Omega |\varphi|^q\,dx,\qquad \varphi\in\mathcal{D}^{1,2}_0(\Omega).
\]
Thus the study performed in this paper is connected to the problem of studying and classifying solutions of the Lane-Emden equation
\[
-\Delta u=|u|^{q-2}\,u,\quad \mbox{ in }\Omega,\qquad u=0,\quad \mbox{ on }\partial\Omega.
\]
We refer to Proposition \ref{prop:free} below, for more details.
\end{oss}
Very little is known on the precise structure of $\mathrm{Spec}(\Omega;q)$. A basic result assures that this is a {\it closed set}, see \cite[Theorem 5.1]{FL}. Moreover, it is unbounded, as it contains a sequence of $q-$eigenvalues diverging to $+\infty$.
Such a sequence is constructed by mimicking the variational characterization \eqref{CFW}. Namely, for every $k\in\mathbb{N}\setminus\{0\}$ one can define
\begin{equation}
\label{CFWLS}
\lambda_{k,LS}(\Omega;q)=\inf_{\mathcal{F}\in \Sigma_k(\Omega;q)}\left\{\max_{\varphi\in\mathcal{F}} \int_\Omega |\nabla\varphi|^2\,dx\right\},
\end{equation}
where 
\[
\Sigma_k(\Omega;q)=\Big\{\mathcal{F}\subset \mathcal{S}_q(\Omega)\, :\, \mathcal{F} \mbox{ compact and symmetric with } \gamma(\mathcal{F})\ge k\Big\},
\]
and $\gamma$ is the {\it Krasnosel'ski\u{\i} genus}, defined by
\[
\gamma(\mathcal{F})=\inf\Big\{k\in\mathbb{N}\setminus\{0\}\, :\, \exists \mbox{ a continuous odd map } \phi:\mathcal{F}\to\mathbb{S}^{k-1}\Big\}.
\] 
Then one has (see \cite[Theorem 5.2]{FL}) 
\[
\mathrm{Spec}_{LS}(\Omega;q):=\{\lambda_{k,LS}(\Omega;q)\}_{k\in\mathbb{N}\setminus\{0\}}\subset\mathrm{Spec}(\Omega;q)\qquad \mbox{ and }\qquad \lim_{k\to \infty}\lambda_{k,LS}(\Omega;q)=+\infty.
\]
The set $\mathrm{Spec}_{LS}(\Omega;q)$ is called {\it Lusternik-Schnirelman $q-$spectrum of the Dirichlet-Laplacian on $\Omega$}.
\par

\subsection{So similar, yet so different!} 
By anticipating some of the conclusions of the paper, we now summarize some peculiar properties of the $q-$spectrum of the Dirichlet-Laplacian of an open set. In particular, we analyze to which extent properties $(\mathcal{E}_1)-(\mathcal{E}_6)$ are still valid for $q\not =2$:
\begin{itemize}
\item[$(\mathcal{E}_{1,q})$] in general, property $(\mathcal{E}_1)$ fails to be true for $1<q<2$, i.e. we can construct an open set $\Omega$ such that $\mathrm{Spec}(\Omega;q)$ is not discrete and has countably many accumulation points (see Example \ref{exa:1}). On the contrary, for $2<q<2^*$, this is an open problem;
\vskip.2cm
\item[$(\mathcal{E}_{2,q})$] essentially nothing is known on the counterpart of $(\mathcal{E}_2)$;
\vskip.2cm
\item[$(\mathcal{E}_{3,q})$] property $(\mathcal{E}_3)$ is still true for $q\not =2$ (see Subsection \ref{sec:first}). However, differently from the case $q=2$, for $1<q<2$ it may happen that the first $q-$eigenvalue is not isolated in the spectrum, i.e. it is an accumulation point of elements of $\mathrm{Spec}(\Omega;q)$ (see Example \ref{exa:2}). For $2<q<2^*$, it is not known whether the first $q-$eigenvalue is isolated or not;
\vskip.2cm
\item[$(\mathcal{E}_{4,q})$] property $(\mathcal{E}_4)$ is still true for $1<q<2$ (see Theorem \ref{teo:simplicity}), but it may fail for $2<q<2^*$ (see Example \ref{exa:dancer}). It is interesting to notice that the set of Example \ref{exa:dancer} has a trivial topology (actually, it is starshaped); 
\vskip.2cm
\item[$(\mathcal{E}_{5,q})$] property $(\mathcal{E}_5)$ fails for $q\not =2$ (see Remark \ref{oss:pezzi}). However, for $2<q<2^*$ the identity \eqref{primino} is still true.
On the contrary, the latter is false for $1<q<2$ (see Example \ref{exa:2});
\vskip.2cm
\item[$(\mathcal{E}_{6,q})$] property $(\mathcal{E}_6)$ fails for $1<q<2$, in the sense that one can exhibit a set for which
\[
\mathrm{Spec}_{LS}(\Omega;q)\not=\mathrm{Spec}(\Omega;q),
\]
(see Example \ref{exa:1}). Here the role of the Krasnosel'ski\u{\i} genus is immaterial, in the sense that the same counter-example still works if we replace the Krasnosel'ski\u{\i} genus with any other index (i.e. {\it $\mathbb{Z}_2-$cohomological index} or {\it Lusternik-Schnirelman category}, just to name a few). We refer the reader to \cite[Chapter II, Section 5]{St} for index theories.
\par
For $2<q<2^*$, this is an open problem.
\end{itemize}
\begin{oss}[Back to $q=2$]
It is useful to keep in mind that for $q=2$, it can be shown that the Lusternik-Schnirelman spectrum coincides with the whole spectrum of the Dirichlet-Laplacian, see for example \cite[Theorem A.2]{BPS}. 
\end{oss}
\begin{oss}[One-dimensional case] Up to now, the whole discussion has concerned the case of dimension $N\ge 2$. In the one-dimensional case, if we take $\Omega=(a,b)\subset\mathbb{R}$, then all the interesting phenomena highlighted above disappear. In particular, by \cite[Theorem II]{Ot} we have that $\mathrm{Spec}((a,b);q)$ is discrete and by \cite[Theorem 4.1]{DM}
\[
\mathrm{Spec}((a,b);q)=\mathrm{Spec}_{LS}((a,b);q),
\]
see also \cite[Theorem 5.3]{FL}. However, even in this case disconnected sets may give weird phenomena, see Remark \ref{oss:dopoesempi} below.

\end{oss}

\subsection{Style of the paper} Where possible, we tried to present proofs which are based on variational principles, rather than on the linearity of the Laplace operator. Also, we tried to keep at a minimal level the regularity assumptions on the sets and the use of regularity for eigenfunctions. For these reasons, many of the results and techniques presented in this paper can be easily generalized to the case of the $p-$Laplacian. This corresponds to replace the Dirichlet integral with the {\it $p-$Dirichlet integral}, i.e.
\[
\varphi\mapsto \int_\Omega |\nabla \varphi|^p\,dx,\qquad \mbox{ for } 1<p<\infty. 
\]
In this case, the equation \eqref{autosalone} must be replaced by its quasilinear version
\[
-\Delta_p u=\lambda\, |u|^{q-2}\,u,\ \mbox{ in } \Omega,\qquad \mbox{ where } -\Delta_p u=-\mathrm{div}(|\nabla u|^{p-2}\,\nabla u).
\]
This eigenvalue--type equation has been introduced in \cite{Ot}.
However, in this case, all the proofs that use a linearization of the equation (see for example Proposition \ref{prop:lin} and Theorem \ref{teo:lin}) should be handled with care and the extension of the relevant results to the $p-$Laplacian are not so straightforward. Some results can be found in \cite{FL}.
\par
In this paper, we preferred to stick to the case of the Laplacian, which is already rich of weird and interesting phenomena...and of open problems, as well.
\par
Finally, {\it \c{c}a va sans dire}, we do not claim that the present work is complete or exhaustive. This paper only reflects the authors' mathematical taste and their knowledge on the problem under consideration

\subsection{Plan of the paper}
In Section \ref{sec:2} we collect some definitions and basic facts. The core of the paper is represented by Sections \ref{sec:3} and \ref{sec:4}, where we separately present our eigenvalue problem, for $1<q<2$ and $q>2$. Both sections have the same structure: we first present the known results, discuss a handful of counter-examples which highlight the main differences with the case $q=2$ and list some open problems. A pair of appendices complement the paper and contribute to make it self-contained.

\begin{ack}
This paper evolved from a set of notes for a talk delivered by the first author at the workshop ``{\it Nonlinear Meeting in Turin 2019\,}''. The organizers Alberto Boscaggin, Francesca Colasuonno and Guglielmo Feltrin are kindly acknowledged. 
\end{ack}

\section{Preliminaries}
\label{sec:2}

\subsection{Notation}
We will indicate by $B_R(x_0)$ the $N-$dimensional open ball with radius $R>0$, centered at $x_0\in\mathbb{R}^N$. When the center is the origin, we will simply write $B_R$.
\par
We define the critical Sobolev exponent
\[
2^*=\left\{\begin{array}{rl}
+\infty, & \mbox{ if } N=2,\\
&\\
\dfrac{2\,N}{N-2}, & \mbox{ if } N\ge 3.
\end{array}
\right.
\]
Occasionally, we will use the celebrated {\it Sobolev inequality} for $\varphi\in \mathcal{D}^{1,2}_0(\Omega)$, i.e.
\begin{equation}
\label{sorbole}
\mathcal{T}_N\,\left(\int_\Omega |\varphi|^{2^*}\,dx\right)^\frac{2}{2^*}\le \int_\Omega |\nabla \varphi|^2\,dx,\qquad \mbox{ for }N\ge 3.
\end{equation}
In dimension $N=2$, the previous inequality does not hold. In this case, we will use the {\it Gagliardo-Nirenberg interpolation inequality}
\begin{equation}
\label{GNS}
\mathcal{T}_{q,\gamma}\,\left(\int_\Omega |\varphi|^\gamma\,dx\right)^\frac{2}{\gamma}\le \left(\int_\Omega |\nabla \varphi|^2\,dx\right)^\frac{\gamma-q}{\gamma}\,\left(\int_\Omega |\varphi|^q\,dx\right)^\frac{2}{\gamma},\qquad \mbox{ for } \gamma>q>1.
\end{equation}

\subsection{Sets, monotonicity and scalings}

\begin{defi}
Let $\Omega\subset\mathbb{R}^N$ be an open set and $1<q<2^*$. We say that $\Omega$ is {\it $q-$admissible} if the embedding $\mathcal{D}^{1,2}_0(\Omega)\hookrightarrow L^q(\Omega)$ is compact.
\end{defi}
Under the condition of $q-$admissibility, one can still produce the Lusternik-Schnirelman $q-$spectrum of the Dirichlet-Laplacian on $\Omega$. Indeed, the existence of this sequence is based on the validity of the so-called {\it Palais-Smale condition} (see \cite[Chpater II, Section 2]{St}), which is assured by the compactness of the embedding $\mathcal{D}^{1,2}_0(\Omega)\hookrightarrow L^q(\Omega)$.
\begin{oss}
\label{oss:qadmissible}
It is well-known that if $\Omega$ has finite $N-$dimensional Lebesgue measure, then it is $q-$admissible for every $1<q<2^*$. However, a set may be $q-$admissible for some $q$, even if its measure is infinite (see \cite[Example 15.5.3]{maz} and \cite[Example 5.2]{braruf} for some examples).
\par
More generally, it is useful to keep in mind the following facts:
\begin{itemize}
\item for $1<q<2$
\[
\mathcal{D}^{1,2}_0(\Omega)\hookrightarrow L^q(\Omega) \mbox{ is continuous }\qquad \Longleftrightarrow\qquad \mathcal{D}^{1,2}_0(\Omega)\hookrightarrow L^q(\Omega) \mbox{ is compact},
\]
see \cite[Theorem 15.6.2]{maz} and \cite[Theorem 1.2]{braruf};
\vskip.2cm
\item for $2<q<2^*$
\[
\mathcal{D}^{1,2}_0(\Omega)\hookrightarrow L^q(\Omega) \mbox{ is compact }\qquad \Longleftrightarrow\qquad \mathcal{D}^{1,2}_0(\Omega)\hookrightarrow L^2(\Omega) \mbox{ is compact},
\]
see \cite[Theorem 15.6.1]{maz}.
\end{itemize}
\end{oss}
By using that for $\Omega'\subset\Omega$ we have $\mathcal{D}^{1,2}_0(\Omega')\subset \mathcal{D}^{1,2}_0(\Omega)$, it is easily seen that
\begin{equation}
\label{mono}
\lambda_{k,LS}(\Omega;q)\le \lambda_{k,LS}(\Omega';q),\qquad \mbox{ for } 1<q<2^* \mbox{ and } k\in\mathbb{N}\setminus\{0\}.
\end{equation}
Moreover, by using the scaling properties of the equation \eqref{autosalone}, we have that if $\lambda\in\mathrm{Spec}(\Omega;q)$, then
\[
t^{N-2-\frac{2}{q}\,N}\,\lambda\in\mathrm{Spec}(t\,\Omega;q),
\]
for every $t>0$.
\subsection{The first $q-$eigenvalue} 
\label{sec:first} By using the definitions of $\lambda_{1,LS}(\Omega;q)$ and of Krasnosel'ski\u{\i} genus, it is easy to see that
\[
\lambda_{1,LS}(\Omega;q)=\min\left\{\int_\Omega |\nabla \varphi|^2\,dx\, :\, \varphi\in \mathcal{S}_q(\Omega)\right\}.
\]
Thus $\lambda_{1,LS}(\Omega;q)$ is the sharp constant for the Poincar\'e--Sobolev inequality
\begin{equation}
\label{poincarello}
\lambda_1(\Omega;q)\,\left(\int_\Omega |\varphi|^q\,dx\right)^\frac{2}{q}\le \int_\Omega |\nabla \varphi|^2\,dx.
\end{equation} 
On the other hand, if $(u,\lambda)$ is a $q-$eigenpair, by testing the weak formulation of \eqref{autosalonebis} with $u$ itself, one obtains
\[
\int_\Omega |\nabla u|^2\,dx=\lambda\,\left(\int_\Omega |u|^q\,dx\right)^\frac{2}{q}.
\]
By recalling \eqref{poincarello}, one then gets
\[
\lambda_{1,LS}(\Omega;q)\le \lambda,\qquad \mbox{ for every } \lambda\in\mathrm{Spec}(\Omega;q).
\]
Thus $\lambda_{1,LS}(\Omega;q)$ is really the first eigenvalue of our eigenvalue problem. For this reason, from now on, when referring to this value we will drop
the uncomfortable subscript $LS$ and simply write $\lambda_1(\Omega;q)$. 
\begin{prop}
\label{lm:constantsign}
Let $1<q<2^*$ and let $\Omega\subset\mathbb{R}^N$ be a $q-$admissible open set. Then any first $q-$eigenfunction of the Dirichlet-Laplacian must have constant sign. 
\end{prop}
\begin{proof}
Let us suppose that $u\in\mathcal{D}^{1,2}_0(\Omega)$ is a first $q-$eigenfunction, such that both $u_+$ and $u_-$ are nontrivial. Here $u_+$ and $u_-$ are the positive and negative parts, respectively. By observing that $|u|\in\mathcal{D}^{1,2}_0(\Omega)$ and that
\[
\int_\Omega |\nabla |u||^2\,dx=\int_\Omega |\nabla u|^2\,dx,\qquad \int_\Omega \Big||u|\Big|^q\,dx=\int_\Omega |u|^q\,dx,
\]
we get that $|u|$ is still a first $q-$eigenfunction. We suppose for simplicity that 
\[
\int_\Omega |u|^q\,dx=1,
\]
thus by minimality, $u$ and $|u|$ solve
\[
\int_\Omega \langle \nabla u,\nabla \varphi\rangle\,dx=\lambda_1(\Omega;q)\,\int_\Omega |u|^{q-2}\,u\,\varphi\,dx,\qquad \mbox{ for every } \varphi\in\mathcal{D}^{1,2}_0(\Omega), 
\]
and
\[
\int_\Omega \langle \nabla |u|,\nabla \varphi\rangle\,dx=\lambda_1(\Omega;q)\,\int_\Omega |u|^{q-1}\,\varphi\,dx,\qquad \mbox{ for every } \varphi\in\mathcal{D}^{1,2}_0(\Omega).
\]
We now observe that $u_+=(|u|+u)/2$, thus by summing the previous equations we get
\[
\begin{split}
\int_\Omega \langle \nabla u_+,\nabla\varphi\rangle\,dx&=\lambda_1(\Omega;q)\,\int_\Omega \left(\frac{|u|^{q-2}\,u+|u|^{q-1}}{2}\right)\,\varphi\,dx\\
&=\lambda_1(\Omega;q)\,\int_\Omega |u|^{q-2}\,\frac{|u|+u}{2}\,\varphi\,dx\\
&=\lambda_1(\Omega;q)\,\int_\Omega |u|^{q-2}\,u_+\,\varphi\,dx,\qquad \mbox{ for every } \varphi\in\mathcal{D}^{1,2}_0(\Omega).
\end{split}
\]
By observing that $|u|^{q-2}\,u_+=u_+^{q-1}$, we get from the previous computation that $u_+$ is a non-negative weak solution of
\[
-\Delta u_+=\lambda_1(\Omega;q)\,u_+^{q-1},\qquad \mbox{ in }\Omega.
\]
In particular, it is a weakly superharmonic function in $\Omega$. On the other hand, the function $u_+$ has to vanish on a set of positive measure, since we are assuming that both $u_+$ and $u_-$ are nontrivial. We now get a contradiction with the minimum principle.
\end{proof}

\begin{oss}
For the case $2<q<2^*$, there is an even simpler proof of the previous fact.
Let us suppose that $u\in\mathcal{D}^{1,2}_0(\Omega)$ is a first $q-$eigenfunction, such that both $u_+$ and $u_-$ are nontrivial. We can assume without loss of generality that
\[
\|u\|_{L^q(\Omega)}=1.
\]
By testing the equation with $u_+$, we obtain
\[
\int_\Omega |\nabla u_+|^2\,dx=\lambda_1(\Omega;q)\,\int_\Omega (u_+)^q\,dx.
\]
Thanks to the normalization taken, we observe that for $q>2$ we have 
\[
\left(\int_\Omega (u_+)^q\,dx\right)^\frac{2}{q}> \int_\Omega (u_+)^q\,dx.
\]
Here we used that $t<t^\alpha$, for $0<t<1$ and $0<\alpha<1$.
This in turn implies that
\[
\frac{\displaystyle\int_\Omega |\nabla u_+|^2\,dx}{\displaystyle\left(\int_\Omega (u_+)^q\,dx\right)^\frac{2}{q}}<\lambda_1(\Omega;q).
\]
This violates the minimality of the value $\lambda_1(\Omega;q)$. This proof does not work for $1<q< 2$.
\end{oss}
\subsection{Miscellaneous stuff}
The following mild regularity result is certainly well-known, this can be found for example in \cite[Theorem 2.2]{FL}. The main focus is on the precise scale-invariant estimate. Observe that the constant entering in the estimate does not depend on the measure $|\Omega|$ of the set. For this reason, we can consider open sets with minimal assumptions. We provide a proof based on the {\it Moser's iteration technique}.
\begin{prop}
\label{prop:Linfty}
Let $1<q<2^*$ and let $\Omega\subset\mathbb{R}^N$ be a $q-$admissible open set. If $U\in\mathcal{D}^{1,2}_0(\Omega)$ is a $q-$eigenfunction with eigenvalue $\lambda$, then $U\in L^\infty(\Omega)$. Moreover, we have the estimate:
\begin{itemize}
\item if $N\ge 3$
\[
\|U\|_{L^\infty(\Omega)}\le C_{N,q}\,\Big(\sqrt{\lambda}\Big)^\frac{2^*}{2^*-q}\,\|U\|_{L^q(\Omega)};
\] 
\item if $N=2$
\[
\|U\|_{L^\infty(\Omega)}\le \left\{\begin{array}{lr}
C_{q}\,\sqrt{\lambda}\,\|U\|_{L^q(\Omega)},& \mbox{ if } 2\le q<2^*,\\
&\\
C_q\,\sqrt{\dfrac{\lambda}{\lambda_1(\Omega;q)}}\,\sqrt{\lambda}\,\|U\|_{L^q(\Omega)},& \mbox{ if } 1<q<2.
\end{array}
\right.
\] 
\end{itemize}
\end{prop}
\begin{proof}
We can assume without loss of generality that $U$ is positive. For simplicity, we set
\[
\Lambda=\lambda\,\|U\|_{L^q(\Omega)}^{2-q}.
\]
Since this result is quite standard, we assume that $U$ is already in $L^\infty(\Omega)$ and just focus on obtaining the claimed a priori estimate. The complete result would follow just by replacing the test function $U^\beta$ below, with $\min\{U,M\}^\beta$ for $M>0$ and then letting $M$ goes to $+\infty$. We leave the details to the reader.
\par
We find it useful to distinguish the cases $q\ge 2$ and $1<q<2$. Indeed, even if the idea of the proof is the same, some computations are different. Moreover, the cases $N\ge 3$ and $N=2$ will need a different treatment, as usual. 
\vskip.2cm\noindent
{\bf Case $2\le q<2^*$}. We test the equation with $\varphi=U^\beta$.
This gives
\[
\frac{4\,\beta}{(\beta+1)^2}\,\int_\Omega \left|\nabla U^\frac{\beta+1}{2}\right|^2\,dx=\Lambda\,\int_\Omega U^{q-1}\,U^\beta\,dx\le \Lambda\,\left(\int_\Omega U^q\,dx\right)^\frac{q-2}{q}\,\left(\int_\Omega U^{\frac{\beta+1}{2}\,q}\,dx\right)^\frac{2}{q}.
\]
If $N\ge 3$, we can now use Sobolev inequality \eqref{sorbole} in the left-hand side, so to get
\begin{equation}
\label{uno}
\mathcal{T}_N\,\left(\int_\Omega U^{\frac{\beta+1}{2}\,2^*}\,dx\right)^\frac{2}{2^*}\le \Lambda\,
\frac{(\beta+1)^2}{4\,\beta}\,\|U\|^{q-2}_{L^q(\Omega)}\,\left(\int_\Omega U^{\frac{\beta+1}{2}\,q}\,dx\right)^\frac{2}{q}.
\end{equation}
We observe that 
\[
\frac{(\beta+1)^2}{4\,\beta}\le \frac{\beta+1}{2},
\]
then, if we set $\vartheta=(\beta+1)/2$, from \eqref{uno} we get
\[
\left(\int_\Omega U^{\vartheta\,2^*}\,dx\right)^\frac{1}{\vartheta\,2^*}\le \left(\frac{\lambda}{\mathcal{S}_N}\right)^\frac{1}{2\,\vartheta}\,
\vartheta^\frac{1}{2\,\vartheta}\,\left(\int_\Omega U^{\vartheta\,q}\,dx\right)^\frac{1}{\vartheta\,q}.
\]
We introduce the sequence of exponents
\[
\vartheta_0=1,\qquad \vartheta_{i+1}=\frac{2^*}{q}\,\vartheta_i=\left(\frac{2^*}{q}\right)^i,\quad i\in\mathbb{N}.
\]
By iterating the previous estimate and observing that
\[
\sum_{i=0}^\infty \frac{1}{\vartheta_i}=\frac{2^*}{2^*-q}\qquad \mbox{ and }\qquad \prod_{i=0}^\infty \vartheta_i^\frac{1}{\vartheta_i}=C_{N,q},
\] 
with a Moser's iteration we get
\[
\|U\|_{L^\infty(\Omega)}\le C\,\Big(\sqrt{\lambda}\Big)^\frac{2^*}{2^*-q}\,\|U\|_{L^q(\Omega)}.
\]
If $N=2$, we need to use \eqref{GNS} in place of \eqref{sorbole}. More precisely, if we take $\gamma=2\,q$ in \eqref{GNS} and use the equation as above, then we get
\[
\begin{split}
C\,\left(\int_\Omega U^{\frac{\beta+1}{2}\,2\,q}\,dx\right)^\frac{2}{2\,q}&\le \left(\int_\Omega \left|\nabla U^\frac{\beta+1}{2}\right|^2\,dx\right)^\frac{1}{2}\,\left(\int_\Omega |U|^{\frac{\beta+1}{2}\,q}\right)^\frac{2}{2\,q}\\
&\le \sqrt{\Lambda\,
\frac{(\beta+1)^2}{4\,\beta}\,\|U\|^{q-2}_{L^q(\Omega)}}\,\left(\int_\Omega U^{\frac{\beta+1}{2}\,q}\,dx\right)^\frac{2}{q},
\end{split}
\]
where $C=C(q)>0$. We can now repeat the same iterative scheme as above, by replacing $2^*$ with $2\,q$. We leave the details to the reader.
\vskip.2cm\noindent
{\bf Case $1<q<2$}. We test again the equation with $\varphi=U^\beta$. As before, we get
\[
\frac{4\,\beta}{(\beta+1)^2}\,\int_\Omega \left|\nabla U^\frac{\beta+1}{2}\right|^2\,dx=\Lambda\,\int_\Omega U^{\beta+q-1}\,dx.
\]
If $N\ge 3$, we use Sobolev inequality in the left-hand side, so to get
\begin{equation}
\label{unoq}
\mathcal{T}_N\,\left(\int_\Omega U^{\frac{\beta+1}{2}\,2^*}\,dx\right)^\frac{2}{2^*}\le \Lambda\,
\frac{(\beta+1)^2}{4\,\beta}\,\int_\Omega U^{\beta+q-1}\,dx.
\end{equation}
If we define the sequence of exponents 
\[
\beta_0=1\qquad \mbox{ and }\qquad \beta_{i+1}=(\beta_i+1)\,\frac{2^*}{2}-(q-1),\quad i\in\mathbb{N},
\]
we can obtain from \eqref{unoq}
\[
\int_\Omega U^{\beta_{i+1}+q-1}\,dx\le \left(\frac{\Lambda}{\mathcal{T}_N}\,
\frac{(\beta_i+1)^2}{4\,\beta_i}\right)^\frac{2^*}{2}\,\left(\int_\Omega U^{\beta_i+q-1}\,dx\right)^\frac{2^*}{2}.
\]
We observe that
\[
\frac{(\beta_i+1)^2}{4\,\beta_i}\le \frac{\beta_i+1}{2}\le 2\,(\beta_i+q-1),
\]
and further define $\vartheta_i=\beta_i+q-1$, with $\vartheta_0=q$. Then we get
\[
\left(\int_\Omega U^{\vartheta_{i+1}}\,dx\right)^\frac{1}{\vartheta_{i+1}}\le \left(\frac{\Lambda}{\mathcal{T}_N}\,
2\,\vartheta_i\right)^{\frac{2^*}{2}\,\frac{1}{\vartheta_{i+1}}}\,\left(\left(\int_\Omega U^{\vartheta_i}\,dx\right)^\frac{1}{\vartheta_i}\right)^{\frac{\vartheta_i}{\vartheta_{i+1}}\frac{2^*}{2}}.
\]
We introduce the notation
\[
Y_i=\|U\|_{L^{\vartheta_i}(\Omega)},\qquad i\in\mathbb{N},
\]
then the previous scheme rewrites as
\[
Y_{i+1}\le \left(\frac{\Lambda}{\mathcal{T}_N}\,
2\,\vartheta_i\right)^{\frac{2^*}{2}\,\frac{1}{\vartheta_{i+1}}}\,Y_i^{\frac{\vartheta_i}{\vartheta_{i+1}}\frac{2^*}{2}}.
\]
We start with $i=0$ and iterate this scheme: after $n$ steps we get
\[
Y_{n+1}\le \left(\frac{2\,\Lambda}{\mathcal{T}_N}\right)^{\frac{1}{\vartheta_{n+1}}\,\sum\limits_{i=0}^n \left(\frac{2^*}{2}\right)^{i+1}}\,\prod_{i=0}^n \left(\vartheta_i\right)^{\frac{1}{\vartheta_{n+1}}\,\left(\frac{2^*}{2}\right)^{n-i+1} } Y_0^{\frac{\vartheta_0}{\vartheta_{n+1}}\,(\frac{2^*}{2})^{n+1}}.
\]
We now observe that by construction\footnote{Indeed, by construction we have 
\[
\vartheta_0=q,\qquad \vartheta_{i+1}=\left(\frac{2^*}{2}\right)^{i+1}\,(\vartheta_i+2-q),
\] 
thus it is not difficult to see that
\[
\vartheta_{i+1}=\left(\frac{2^*}{2}\right)^{i+1}\,\vartheta_0+(2-q)\,\frac{2^*}{2}\,\sum_{k=0}^i \left(\frac{2^*}{2}\right)^k,\qquad \mbox{ for }i\in\mathbb{N}.
\]}
\[
\vartheta_i\sim \left(\frac{2^*}{2}\right)^i\,\left[q+(2-q)\,\frac{2^*}{2^*-2}\right],\qquad \mbox{ for } i \to \infty,
\]
thus we get
\[
\lim_{n\to\infty}\left(\frac{2\,\Lambda}{\mathcal{T}_N}\right)^{\frac{1}{\vartheta_{n+1}}\,\sum\limits_{i=0}^n \left(\frac{2^*}{2}\right)^{i+1}}=\left(\frac{2\,\Lambda}{\mathcal{T}_N}\right)^{\frac{1}{q+(2-q)\,\frac{2^*}{2^*-2}}\,\frac{2^*}{2^*-2}},
\]
and
\[
\lim_{n\to\infty} Y_0^{\frac{\vartheta_0}{\vartheta_{n+1}}\,(\frac{2^*}{2})^{n+1}}=Y_0^\frac{q}{q+(2-q)\,\frac{2^*}{2^*-2}}.
\]
Moreover, we have 
\[
\begin{split}
\lim_{n\to\infty} \prod_{i=0}^n \left(\vartheta_i\right)^{\frac{1}{\vartheta_{n+1}}\,\left(\frac{2^*}{2}\right)^{n-i+1} }&=\lim_{n\to\infty} \exp\left(\frac{1}{\vartheta_{n+1}}\,\left(\frac{2^*}{2}\right)^{n+1}\,\sum_{i=0}^n \left(\frac{2^*}{2}\right)^{-i}\,\log \vartheta_i\right)<+\infty,
\end{split}
\]
again thanks to the asymptotic behaviour of $\vartheta_i$. In conclusion, we get
\[
Y_\infty\le C\,\left(\Lambda^\frac{2^*}{2^*-2}\,Y_0^q\right)^\frac{1}{q+(2-q)\,\frac{2^*}{2^*-2}},
\]
for a constant $C=C(N,q)>0$. By recalling the definition of $\Lambda$ and $Y_i$, this is the same as
\[
\|U\|_{L^\infty(\Omega)}\le C\,\Big(\sqrt{\lambda}\Big)^\frac{2^*}{2^*-q}\,\|U\|_{L^q(\Omega)}.
\]
This concludes the proof.
\par
The case $N=2$ needs the following modification. We first observe that by coupling \eqref{GNS} with 
\[
\lambda_1(\Omega;q)\,\left(\int_\Omega |\varphi|^q\,dx\right)^\frac{2}{q}\le \int_\Omega |\nabla \varphi|^2\,dx,\qquad \mbox{ for every } \varphi\in\mathcal{D}^{1,2}_0(\Omega),
\]
we get for $\gamma>q$
\[
\mathcal{T}_{q,\gamma}\,\Big(\lambda_1(\Omega;q)\Big)^\frac{q}{\gamma}\,\left(\int_\Omega |\varphi|^\gamma\right)^\frac{2}{\gamma}\le \int_\Omega |\nabla \varphi|^2\,dx,\qquad \mbox{ for every } \varphi\in\mathcal{D}^{1,2}_0(\Omega).
\]
We use this estimate with $\gamma=2\,q$, thus in place of \eqref{unoq} we now get
\[
\begin{split}
C\,\Big(\lambda_1(\Omega;q)\Big)^\frac{1}{2}\,\left(\int_\Omega U^{\frac{\beta+1}{2}\,2\,q}\,dx\right)^\frac{2}{2\,q}&\le \int_\Omega \left|\nabla U^\frac{\beta+1}{2}\right|^2\,dx\\
&\le \Lambda\,
\frac{(\beta+1)^2}{4\,\beta}\,\int_\Omega U^{\beta+q-1}\,dx,
\end{split}
\]
with $C>0$. We can repeat the iterative scheme as above, again with $2\,q$ in place of $2^*$.
\end{proof}

The following result is important in order to study the set $\mathrm{Spec}(\Omega;q)$ for a disconnected set $\Omega$. It is contained in \cite[Corollary 2.2]{BF}: the result in \cite{BF} is stated for $1<q<2$ only, but a closer inspection of the proof reveals that it still works for $q>2$.
\begin{prop}[The ``spin formula'' for disconnected sets]
\label{prop:crucial}
Let $1<q<2^*$ with $q\not =2$ and let $\Omega\subset\mathbb{R}^N$ be a $q-$admissible open set. Let $\#\in\mathbb N\cup\{+\infty\}$
and suppose that
\[
\Omega=\bigcup_{i=1}^\# \Omega_i,
\]
with $\Omega_i\subset \mathbb{R}^N$ being an open set, such that $\mathrm{dist}(\Omega_i,\Omega_j)>0$, for $i\not =j$. Then $\lambda$ is a $q-$eigenvalue of $\Omega$ if and only if it is of the form
\begin{equation}
\label{rap_gen}
\lambda=\left[\displaystyle\sum_{i=1}^\#\left(\frac{\delta_i}{\lambda_i}\right)^\frac{q}{2-q}\right]^\frac{q-2}{q}\qquad \mbox{ for some $q-$eigenvalue $\lambda_i$ of $\Omega_i$},
\end{equation}
where the {\rm spin coefficients} $\delta_i$ are such that
\[
\delta_i\in\{0,1\}\qquad \mbox{ and }\qquad \sum_{i=1}^\#\delta_i\not =0.
\]
Moreover, if we set 
\[
|\alpha_i|=\displaystyle\left(\frac{\lambda}{\lambda_i}\right)^\frac{1}{2-q},
\]
each corresponding $q-$eigenfunction $U$ of $\Omega$ has the form
\[
U(x)=C\,\sum_{i=1}^\#\delta_i\,\alpha_i\,u_i(x),
\]
where $C\in\mathbb{R}$ and $u_i\in\mathcal{D}^{1,2}_0(\Omega)$ is $q-$eigenfunction of $\Omega_i$ with unit $L^q$ norm corresponding to $\lambda_i$.
\end{prop}
\begin{oss}
\label{oss:pezzi}
By fixing $j$ and choosing 
\[
\delta_i=\left\{\begin{array}{cc}
1,& \mbox{ if } i=j,\\
0,& \mbox{ if } i\not =j,
\end{array}
\right.
\]
we get from the Proposition \ref{prop:crucial} that 
\[
\mathrm{Spec}(\Omega_j;q)\subset \mathrm{Spec}\left(\bigcup_{i=1}^\# \Omega_i;q\right)\qquad \mbox{ and thus }\qquad \bigcup_{j=1}^\#\mathrm{Spec}(\Omega_j;q)\subset \mathrm{Spec}\left(\bigcup_{i=1}^\# \Omega_i;q\right).
\]
However, differently from the case $q=2$, the collection of all these spectra {\it does not exhaust the whole spectrum of}
\[
\Omega=\bigcup_{i=1}^\# \Omega_i.
\]
This will be clear from Examples \ref{exa:1} and \ref{exa:2} below.
\end{oss}
The following result is straightforward, the details are left to the reader.
\begin{prop}
\label{prop:free}
Let $1<q<2^*$ with $q\not =2$, we define the free functional
\[
\mathfrak{F}_q(\varphi)=\frac{1}{2}\,\int_\Omega |\nabla \varphi|^2\,dx-\frac{1}{q}\,\int_\Omega |\varphi|^q\,dx,\qquad \mbox{ for } \varphi\in\mathcal{D}^{1,2}_0(\Omega).
\]
Then we have:
\begin{enumerate}
\item if $(u,\lambda)$ is a $q-$eigenpair, the function 
\[
U=\lambda^\frac{1}{q-2}\,\frac{u}{\|u\|_{L^q(\Omega)}},
\] 
is a critical point of $\mathfrak{F}_q$, with critical value 
\[
\left(\frac{1}{2}-\frac{1}{q}\right)\,\lambda^\frac{q}{q-2};
\]
\vskip.2cm
\item if $U\in\mathcal{D}^{1,2}_0(\Omega)$ is a critical point of $\mathfrak{F}_q$, then
\[
(U,\|U\|_{L^q(\Omega)}^{q-2})
\]
is a $q-$eigenpair.
\end{enumerate}
\end{prop}

Finally, we will use the following classical result. For the proof, we refer for example to \cite[Lemma 1.4, Chapter III]{St}. We recall that this is based on testing the equation with the function $\langle x,\nabla u\rangle$ and then using some integrations by parts.
\begin{prop}[Rellich-Pohozaev identity]
\label{lm:RP}
Let $\Omega\subset\mathbb{R}^N$ be a bounded open set, satisfying (at least) one of the following two conditions: 
\begin{itemize}
\item $\Omega$ is of class $C^{1,1}$;
\vskip.2cm
\item $\Omega$ is convex.
\end{itemize}
Let $1<q<2^*$, if $(u,\lambda)$ is a $q-$eigenpair, then we have
\[
\lambda\,\left(\int_\Omega |u|^q\,dx\right)^\frac{2}{q}=C_{q,N}\,\int_{\partial\Omega} |\nabla u|^2\,\langle x,\nu_\Omega\rangle\,d\mathcal{H}^{N-1},
\]
where
\[ 
C_{q,N}=\left\{\begin{array}{ll}
\dfrac{q}{4},& \mbox{ if } N=2,\\
&\\
\dfrac{1}{2\,N}\,\dfrac{2^*\,q}{2^*-q},& \mbox{ if }N\ge 3.
\end{array}
\right.
\]
\end{prop}

\section{The sub-homogeneous case $1<q<2$}
\label{sec:3}

\subsection{Results}

\begin{teo}[Simplicity]
\label{teo:simplicity}
Let $1<q<2$ and let $\Omega\subset\mathbb{R}^N$ be a $q-$admissible open connected set. Then $\lambda_1(\Omega;q)$ is simple.
\end{teo}
\begin{proof}
There are different proofs of this fact. We could for example exploit the so-called {\it hidden convexity}, i.e. the fact that the Dirichlet integral is convex along curves of the form
\[
\sigma_t(x)=\left((1-t)\,u_0(x)^q+t\,u_1(x)^q\right)^\frac{1}{q},\qquad \mbox{ for }  t\in[0,1],
\]
whenever $1\le q\le 2$ and $u_0,u_1$ are nonnegative, see \cite[Proposition 4]{Ka} and also \cite[Proposition 2.6]{BFkodai} for a more general statement. Moreover, convexity is strict on functions satisfying the minimum principle.
\par
Here we prefer to use a trick introduced by Brezis and Oswald in \cite{BO}, which is based on the\footnote{As explained in \cite[Section 3]{BFkodai}, this proof and the one based on the hidden convexity are essentially the same.} {\it Picone's inequality}. The latter assures that
\begin{equation}
\label{picone}
\left\langle \nabla \psi,\nabla\left(\frac{\varphi^2}{\psi}\right)\right\rangle\le |\nabla \varphi|^2,
\end{equation}
for every pair of differentiable functions $\psi,\varphi$, with $\varphi\ge 0$ and $\psi>0$. Let us suppose that $u,v\in\mathcal{D}^{1,2}_0(\Omega)$ are first $q-$eigenfunctions. By Proposition \ref{lm:constantsign}, we know that $u,v$ have constant sign, we can suppose them to be positive.  For simplicity, we further assume that they both have unit $L^q(\Omega)$ norm. If we test the equation for $u$ with
\[
\varphi=\frac{v^2}{u+\varepsilon},\qquad \mbox{ for } \varepsilon>0,
\]
we get
\[
\lambda_1(\Omega)\,\int_\Omega u^{q-1}\,\frac{v^2}{u+\varepsilon}=\int_\Omega \left\langle \nabla u,\nabla \frac{v^2}{u+\varepsilon}\right\rangle\,dx\le \int_\Omega |\nabla v|^2\,dx=\lambda_1(\Omega).
\]
Observe that we used \eqref{picone} above.
By taking the limit as $\varepsilon$ goes to $0$ and using Fatou's Lemma, we get
\begin{equation}
\label{grande}
\int_\Omega u^{q-2}\,v^2\,dx\le 1.
\end{equation}
We can repeat the above computations, by exchanging the roles of $u$ and $v$. This also gives
\begin{equation}
\label{grande2}
\int_\Omega v^{q-2}\,u^2\,dx\le 1.
\end{equation}
We now observe that for every $a,b> 0$
\[
(a^{q-2}-b^{q-2})\,(a^2-b^2)\le 0,
\]
and the inequality sign is strict, whenever $a\not= b$. By taking $a=u(x)$ and $b=v(x)$ and integrating, we get
\[
\begin{split}
0&\ge \int_\Omega (u^{q-2}-v^{q-2})\,(u^2-v^2)\,dx\\
&=\int_\Omega u^q\,dx+\int_\Omega v^q\,dx-\left(\int_\Omega u^{q-2}\,v^2\,dx+\int_\Omega v^{q-2}\,u^2\,dx\right)\\
&\ge \int_\Omega u^q\,dx+\int_\Omega v^q\,dx-2=0. 
\end{split}
\]
In the last inequality, we used \eqref{grande} and \eqref{grande2}. Thus we get
\[
\int_\Omega (u^{q-2}-v^{q-2})\,(u^2-v^2)\,dx=0,
\]
which in turn implies that $u=v$ in $\Omega$.
\end{proof}

\begin{teo}[Positive eigenfunctions]
\label{teo:positive}
Let $1<q<2$ and let $\Omega\subset\mathbb{R}^N$ be a $q-$admissible open connected set. If $\lambda\in\mathrm{Spec}(\Omega;q)$ admits a constant sign eigenfunction, then $\lambda=\lambda_1(\Omega;q)$. 
\end{teo}
\begin{proof}
Here as well, there are various proofs of this fact. The quickest one is probably the one based on the following {\it generalized Picone's inequality} (see \cite[Proposition 2.9]{BFkodai})
\begin{equation}
\label{piccone}
\left\langle \nabla \psi,\nabla\left(\frac{\varphi^q}{\psi^{q-1}}\right)\right\rangle\le |\nabla \varphi|^q\,|\nabla \psi|^{2-q},
\end{equation}
which holds for every pair of differentiable functions $\psi,\varphi$, with $\varphi\ge 0$ and $\psi>0$.
\par
We first observe that we only need to prove that $\lambda\le \lambda_1(\Omega;q)$. Then we take $u$ to be a positive $q-$eigenfunction corresponding to $\lambda$ and $U$ to be a first positive $q-$eigenfunction. As usual, we take the normalization
\[
\int_\Omega u^q\,dx=\int_\Omega U^q\,dx=1.
\]
By testing the equation with the function
\[
\varphi=\frac{U^q}{(\varepsilon+u)^{q-1}},
\]
we have
\[
\begin{split}
\lambda\,\int_\Omega u^{q-1}\,\frac{U^q}{(\varepsilon+u)^{q-1}}\,dx&=\int_\Omega \left\langle \nabla u,\nabla\frac{U^q}{(\varepsilon+u)^{q-1}}\right\rangle\,dx\\&\le \int_\Omega |\nabla U|^q\,|\nabla u|^{2-q}\,dx,
\end{split}
\]
thanks to \eqref{piccone}. We also used that $\nabla u=\nabla (\varepsilon+u)$. If we use H\"older's inequality in the last integral and recall that
\[
\int_\Omega |\nabla U|^2\,dx=\lambda_1(\Omega;q),\qquad \int_\Omega |\nabla u|^2\,dx=\lambda, 
\]
we thus obtain
\[
\lambda\,\int_\Omega u^{q-1}\,\frac{U^q}{(\varepsilon+u)^{q-1}}\,dx\le \Big(\lambda_1(\Omega;q)\Big)^\frac{q}{2}\,\lambda^\frac{2-q}{2},
\]
that is 
\[
\lambda\,\left(\int_\Omega u^{q-1}\,\frac{U^q}{(\varepsilon+u)^{q-1}}\,dx\right)^\frac{2}{q}\le \lambda_1(\Omega;q).
\]
If we now take the limit as $\varepsilon$ goes to $0$, use Fatou's Lemma and the fact that $u>0$ by the minimum principle, we finally get the desired result.
\end{proof}
The next result assures that there exists a gap in $\mathrm{Spec}(\Omega;q)$ after the first $q-$eigenvalue, provided the set $\Omega$ is sufficiently ``nice''. As we will show in the next subsection, the assumptions are optimal, in a sense.
This is taken from \cite{BDF}, which actually contains a slightly more general result.
\begin{teo}[Isolation]
\label{teo:isolation}
Let $\Omega\subset\mathbb{R}^N$ be an open bounded set, having a finite number of connected components. Let us suppose that each connected component has a Lipschitz boundary and satisfies the uniform interior ball condition. 
\par
Then for every $1<q<2$, the first eigenvalue $\lambda_1(\Omega;q)$ is isolated. In other words, if we define
\[
\inf\{\lambda \in\mathrm{Spec}(\Omega;q)\, :\, \lambda>\lambda_{1}(\Omega;q)\},
\]
then this is a $q-$eigenvalue, larger than $\lambda_1(\Omega;q)$.
\end{teo}
\begin{oss}
Observe that the infimum above is actually a minimum, due to the closedness of the spectrum.
\end{oss}

\subsection{Counter-examples}
In general, for $1<q<2$ the set $\mathrm{Spec}(\Omega;q)$ is not discrete. This is the content of the next example, taken from \cite{BF}.
\begin{exa}[The spectrum may not be discrete]
\label{exa:1}
Let $1<q<2$ and $0<r\le R$. We take two disjoint balls $B_r(x_0)$ and $B_R(y_0)$  and set 
\[
\mathcal{B}=B_R(x_0)\cup B_r(y_0).
\] 
Then
\begin{equation}
\label{diversi!}
\mathrm{Spec}_{LS}(\mathcal{B};q)\not =\mathrm{Spec}(\mathcal{B};q).
\end{equation}
Moreover, the set $\mathrm{Spec}(\mathcal{B};q)$ has (at least) countably many accumulation points.
\end{exa}
\begin{proof}
This is based on the ``spin formula'' \eqref{rap_gen}. By using this, we can show that every variational $q-$variational eigenvalue of $B_r(x_0)$ or $B_r(x_0)$ is actually an accumulation point for the $q-$spectrum. Indeed, take for example the $k-$th variational $q-$eigenvalue
\[
\lambda_{k,LS}(B_R(x_0);q),
\]
defined in \eqref{CFWLS}.
By Remark \ref{oss:pezzi}, we know that $\lambda_{k,LS}(B_R(x_0);q)\in\mathrm{Spec}(\mathcal{B};q)$. We now take the sequence 
\[
\Lambda_{n,k}=\left[\displaystyle\left(\frac{1}{\lambda_{k,LS}(B_R(x_0);q)}\right)^\frac{q}{2-q}+\left(\frac{1}{\lambda_{n,LS}(B_r(y_0);q)}\right)^\frac{q}{2-q}\right]^\frac{q-2}{q}.
\]
By formula \eqref{rap_gen}, we know that this is a $q-$eigenvalue of $\mathcal{B}$. Moreover, by using that $\lambda_{n,LS}(B_r(y_0);q)$ diverges to $+\infty$ and that\footnote{Here we use that $1<q<2$. For $2<q<2^*$. we would have
\[
\lim_{s\to 0} \left(t^\frac{q}{2-q}+s^\frac{q}{2-q}\right)^\frac{q-2}{q}=+\infty.
\]}
\[
\lim_{s\to 0} \left(t^\frac{q}{2-q}+s^\frac{q}{2-q}\right)^\frac{q-2}{q}=t,
\]
we get
\[
\lim_{n\to\infty}\Lambda_{n,k}=\lambda_{k,LS}(B_R(x_0);q),
\]
as desired.
\end{proof}
We have seen in Theorem \ref{teo:isolation} that the first $q-$eigenvalue is isolated for $1<q<2$, provided that the set has a finite number of smooth connected components. If we drop the restriction on the number of connected components, the isolation fails. This example is taken from \cite{BF}, as well.
\begin{exa}[The first eigenvalue may not be isolated]
\label{exa:2}
Let $1<q<2$ and let $\{r_i\}_{i\in\mathbb{N}}\subset \mathbb{R}$ be a sequence of strictly positive numbers, such that
\begin{equation}
\label{compatto}
\sum_{i=0}^\infty r_i^{N+\frac{2\,q}{2-q}}<+\infty.
\end{equation}
We take a sequence of points $\{x_i\}_{i\in\mathbb{N}}\subset\mathbb{R}^N$ such that the balls $B_{r_i}(x_i)$ are pairwise disjoint. Accordingly, we set
\[
\mathcal{T}=\bigcup_{i=0}^\infty B_{r_i}(x_i).
\]
Then 
\[
\mathrm{Spec}_{LS}(\mathcal{T};q)\not =\mathrm{Spec}(\mathcal{T};q).
\]
and the set $\mathrm{Spec}(\mathcal{T};q)$ has (at least) countably many accumulation points. Moreover, the first eigenvalue $\lambda_1(\mathcal{T};q)$ is not isolated. 
\end{exa}
\begin{proof}
The hypothesis \eqref{compatto} guarantees that the embedding $\mathcal{D}^{1,2}_0(\mathcal{T})\hookrightarrow L^q(\mathcal{T})$ is compact, see \cite[Example 5.2]{braruf}.
The first part of the statement is exactly as in the previous example. Let us prove that the first eigenvalue is not isolated.
By the ``spin formula'' \eqref{rap_gen}, we know that
\begin{equation}
\label{primospin}
\lambda_1(\mathcal{T};q)=\left[\displaystyle\sum_{i=1}^\infty\left(\frac{\delta_i}{\lambda_i}\right)^\frac{q}{2-q}\right]^\frac{q-2}{q}\qquad \mbox{ for some $q-$eigenvalue $\lambda_i$ of $B_{r_i}(x_i)$},
\end{equation}
and some $\delta_i$ such that
\[
\delta_i\in\{0,1\}\qquad \mbox{ and }\qquad \sum_{i=1}^\infty\delta_i\not =0.
\]
We now observe that for every $n\in\mathbb{N}\setminus\{0\}$ the function 
\begin{equation}
\label{funzioncina}
(t_1,\dots,t_n)\mapsto \left(\frac{1}{t_1^\frac{q}{2-q}+\dots+t_n^\frac{q}{2-q}}\right)^\frac{2-q}{q},\qquad t_1,\dots,t_n>0,
\end{equation}
is monotone decreasing with respect to each variable. Since $\lambda_1(\mathcal{T};q)$ has to be the smallest eigenvalue, this means that we must take 
\[
\delta_i=1\qquad \mbox{ and }\qquad \lambda_i=\lambda_1(B_{r_i}(x_i);q),\qquad \mbox{ for every } i\in\mathbb{N},
\]
in order to make \eqref{primospin} as small as possible\footnote{Here we crucially use that $1<q<2$. For $2<q<2^*$, the function \eqref{funzioncina} can be written as 
\[
(t_1,\dots,t_n)\mapsto \left(\frac{1}{t_1^\frac{q}{q-2}}+\dots+\frac{1}{t_n^\frac{q}{q-2}}\right)^\frac{q-2}{q},\qquad t_1,\dots,t_n>0,
\]
thus in order to make \eqref{primospin} as small as possible, we have to take all $\delta_i=0$ except one (this corresponds to let all $t_i$ goes to $+\infty$, except one). For this reason, this example does not work for $2<q<2^*$.}.
Thus we have 
\[
\lambda_1(\mathcal{T};q)=\left[\displaystyle\sum_{i=1}^\infty\left(\frac{1}{\lambda_1(B_{r_i}(x_i);q)}\right)^\frac{q}{2-q}\right]^\frac{q-2}{q}.
\]
In other words, any first $q-$eigenfunction of $\mathcal{T}$ must be supported on the whole set $\mathcal{T}$.
On the other hand, still by the ``spin formula'' we have that 
\begin{equation}
\label{sequence}
\Lambda_k=\left[\displaystyle\sum_{i=1}^k\left(\frac{1}{\lambda_1(B_{r_i}(x_i);q)}\right)^\frac{q}{2-q}\right]^\frac{q-2}{q}>\lambda_1(\mathcal{T};q),
\end{equation}
is a $q-$eigenvalue of $\mathcal{T}$. By observing that 
\[
\lim_{k\to\infty} \Lambda_k=\lambda_1(\mathcal{T};q),
\]
we get the desired conclusion.
\end{proof}

\begin{oss}
\label{oss:dopoesempi}
In the previous examples, we took for simplicity disjoint unions of balls. Of course, the very same examples work by taking disjoint unions of generic open bounded sets. Also observe that the previous examples work for dimension $N=1$, as well. This implies that \cite[Theorem II]{Ot} and \cite[Theorems 3.1 and 4.1]{DM} fail to be true if $\Omega\subset\mathbb{R}$ is a disjoint union of intervals. Thus, even in dimension $N=1$, we have examples of sets such that
\[
\mathrm{Spec}(\Omega;q) \mbox{ is not discrete }\qquad \mbox{ and }\qquad \mathrm{Spec}(\Omega;q)\not=\mathrm{Spec}_{LS}(\Omega;q),
\]
for $1<q<2$.
\end{oss}

\subsection{Open problems} We list here some questions for the case $1<q<2$ which, to the best of our knowledge, are open.

\begin{open}
On a ``good'' open set $\Omega\subset\mathbb{R}^N$, the $q-$spectrum is discrete and 
\[
\mathrm{Spec}(\Omega;q)=\mathrm{Spec}_{LS}(\Omega;q).
\]
\end{open}

\begin{open}
Whenever $\lambda_1(\Omega;q)$ is isolated, find a variational characterization of the second eigenvalue
\[
\inf\{\lambda \in\mathrm{Spec}(\Omega;q)\, :\, \lambda>\lambda_{1}(\Omega;q)\}.
\]
Does this coincide with $\lambda_{2,LS}(\Omega;q)$ defined in \eqref{CFWLS}?
\end{open}

\section{The super-homogeneous case $2<q<2^*$}
\label{sec:4}

\subsection{Results}
In this case, the situation for the first $q-$eigenvalue abruptly changes. As we will see, Theorems \ref{teo:simplicity} and \ref{teo:positive} do not hold anymore.

\begin{teo}[Simplicity in a ball]
\label{teo:simpleball}
Let $R>0$, then for every $2<q<2^*$ the first eigenvalue $\lambda_1(B_R(0);q)$ is simple.
\end{teo}
\begin{proof}
We take $U$ a first $q-$eigenfunction, with unit $L^q$ norm. Thanks to Proposition \ref{lm:constantsign}, we can suppose that $U\ge 0$. We now divide the proof into three steps.
\vskip.2cm\noindent
{\bf Step 1: reduction to radial functions.} Here we use the same argument of \cite[Theorem 3, point a)]{Ka}.
We consider the {\it radially symmetric decreasing rearrangement} $U^*$ of $U$. This is the unique radially symmetric function such that
\[
|\{x\in B_R(0)\, :\, U(x)>t\}|=|\{x\in B_R(0)\, :\, U^*(x)>t\}|.
\]
This in particular implies that
\[
1=\int_{B_R(0)} U^q\,dx=\int_{B_R(0)} (U^*)^q\,dx.
\]
By using the celebrated {\it P\'olya-Szeg\H{o} principle}, we know that $U^*\in\mathcal{D}^{1,2}_0(\Omega)$ and that 
\[
\lambda_1(B_R(0);q)=\int_{B_R(0)} |\nabla U|^2\,dx\ge \int_{B_R(0)} |\nabla U^*|^2\,dx.
\]
The last two displays shows that $U^*$ is still a first $q-$eigenfunction, thus actually
\[
\int_{B_R(0)} |\nabla U|^2\,dx= \int_{B_R(0)} |\nabla U^*|^2\,dx.
\]
We now want to appeal to the characterization of equality cases in the P\'olya-Szeg\H{o} principle. For this, we observe that by Hopf's boundary Lemma, there exists $0<r<R$ such that
\[
|\nabla U|\ge c>0,\qquad \mbox{ for }r<|x|<R.
\]
This shows that 
\[
\{x\in B_R(0)\, :\, |\nabla U(x)|=0\}\Subset B_R(0).
\]
Moreover, by Lemma \ref{lm:criticalset}, we have that $|\nabla U|\not =0$ almost everywhere in $B_R(0)$. In conclusion, we obtain that
\[
|\{x\in B_R(0)\, :\, |\nabla U(x)|=0\}|=0.
\]
We can now use \cite[Theorem 1.1]{BZ} to infer that $U=U^*$. Thus, any positive first $q-$eigenfunction must be radially simmetric decreasing.
\vskip.2cm\noindent
{\bf Step 2: reduction to a Cauchy problem.}
We now know that any first $q-$eigenfunction $U$ with unit $L^q$ norm has the form $U(x)=u(|x|)$. By using spherical coordinates, the function $u$ must solve the one-dimensional problem 
\[
\eta_{q}(R)=\min_{\varphi\in W^{1,2}((0,R))} \left\{\int_0^R |\varphi'|^2\,\varrho^{N-1}\,d\varrho\, :\, \int_0^R |\varphi|^q\,\varrho^{N-1}\,d\varrho=\frac{1}{N\,\omega_N},\ \varphi(R)=0\right\}.
\]
Moreover, it holds 
\[
\lambda_1(B_R(0);q)=\omega_N^\frac{q-2}{q}\,\eta_q(R).
\]
We are thus lead to show that the previous one-dimensional problem has a unique positive minimizer. By using the Rellich-Pohozaev identity (Lemma \ref{lm:RP}), we have 
\[
\lambda_1(B_R(0);q)=C_{q,N}\,R\,\int_{\partial B_R(0)} |\nabla U|^2\,d\mathcal{H}^{N-1}.
\] 
Since $U(x)=u(|x|)$ is radially symmetric and decreasing, this implies that 
\[
u'(R)=\sqrt{\frac{\lambda_1(B_R(0);q)}{N\,\omega_N\,R^N}\,\frac{1}{C_{q,N}}}=:\mathcal{C}.
\]
Observe that the last is a universal constant, in the sense that it does not depend on $u$.
Thus $u$ must be a positive solution of the following ``backward'' Cauchy problem
\begin{equation}
\label{cauchy}
\left\{\begin{array}{rcll}
-(\varrho^{N-1}\,u')'&=&\Big(\omega_N^\frac{q-2}{q}\,\eta_q(R)\Big)\,\varrho^{N-1}\,u^{q-1},& \mbox{ in } (0,R)\\
u(R)&=&0\\
u'(R)&=&\mathcal{C}.
\end{array}
\right.
\end{equation}
{\bf Step 3: uniqueness for the Cauchy problem.}
We claim that \eqref{cauchy} has a unique positive solution. In order to prove this, we adapt the argument of \cite[Lemma 3.3]{FL}. Thus, we first observe that $u$ is a solution of \eqref{cauchy} if and only if 
\[
u(\varrho)=-\int_\varrho^R \frac{\mathcal{C}\,R^{N-1}}{t^{N-1}}\,dt- \int_\varrho^R \frac{1}{t^{N-1}}\,\left(\int_t^R\Big(\omega_N^\frac{q-2}{q}\,\eta_q(R)\Big)\,\tau^{N-1}\,u^{q-1}(\tau)\,d\tau \right)\,dt.
\]
Let us now suppose that $u_1$ and $u_2$ are two distinct positive solutions of \eqref{cauchy}. We thus get
\begin{equation}
\label{stimabase}
|u_1(\varrho)-u_2(\varrho)|\le \int_\varrho^R \frac{\omega_N^\frac{q-2}{q}\,\eta_q(R)}{t^{N-1}}\left(\int_t^R \tau^{N-1}\,|u_1^{q-1}(\tau)-u_2^{q-1}(\tau)|\,d\tau\right)\,dt.
\end{equation}
By using that 
\[
|a^{q-1}-b^{q-1}|\le (q-1)\,(a^{q-2}+b^{q-2})\,|a-b|,
\]
and recalling the uniform $L^\infty$ estimate for $q-$eigenfunctions (i.e. Proposition \ref{prop:Linfty}), for every $0<r<R$ we get from the previous estimate
\begin{equation}
\label{annullati}
\|u_1-u_2\|_{L^\infty([r,R])}\le C\,\|u_1-u_2\|_{L^\infty([r,R])}\,\int_r^R \frac{1}{t^{N-1}}\left(\int_t^R \tau^{N-1}\,d\tau\right)\,dt,
\end{equation}
where $C>0$ is a uniform constant. We now observe that 
\[
\lim_{r\to R}\int_r^R \frac{1}{t^{N-1}}\left(\int_t^R \tau^{N-1}\,d\tau\right)\,dt=0,
\]
thus, by choosing $r$ sufficiently close to $R$, we can have 
\[
C\,\int_r^R \frac{1}{t^{N-1}}\left(\int_t^R \tau^{N-1}\,d\tau\right)\,dt\le \frac{1}{2}.
\]
By using this in \eqref{annullati}, we get that $u_1=u_2$ in $[r,R]$, for $R-r$ small enough. We now set
\[
r_0=\inf\Big\{r\in (0,R)\, :\, u_1=u_2 \mbox{ on } [r,R]\Big\}.
\]
By the previous argument, we know that $r_0<R$. We assume by contradiction that $r_0>0$. Thus from \eqref{stimabase} we get for every $0<\varrho\le r_0$
\[
|u_1(\varrho)-u_2(\varrho)|\le \int_\varrho^{r_0} \frac{\omega_N^\frac{q-2}{q}\,\eta_q(R)}{t^{N-1}}\left(\int_t^{r_0} \tau^{N-1}\,|u_1^{q-1}(\tau)-u_2^{q-1}(\tau)|\,d\tau\right)\,dt,
\]
where we used that $u_1(\tau)=u_2(\tau)$ for $\tau\in[r_0,R]$. By choosing $r<r_0$, taking the supremum over the interval $[r,r_0]$ and proceeding as before, we get
\[
\|u_1-u_2\|_{L^\infty([r,r_0])}\le C\,\|u_1-u_2\|_{L^\infty([r,r_0])}\,\int_r^{r_0} \frac{1}{t^{N-1}}\left(\int_t^{r_0} \tau^{N-1}\,d\tau\right)\,dt.
\]
By taking $r-r_0$ sufficiently small, we can then claim that there exists $r<r_0$ such that $u_1=u_2$ on $[r,R]$. This violates the definition of $r_0$, thus we get $r_0$, as desired. This finally proves that \eqref{cauchy} has a unique solution.
\par
Thus the proof of the theorem is complete.
\end{proof}
\begin{oss}
One could also use the classical symmetry result \cite[Theorem 1]{GNN} by Gidas, Ni and Nirenberg, to achieve {\bf Step 1} of the previous proof. Here we preferred to stick to a more variational argument. 
\par
The previous result can be found in \cite[Theorem 2 and Corollary 1]{KL}. However, the proof there is slightly different: namely, in order to show uniqueness for the relevant ODE, the authors in \cite{KL} use a Kelvin--type transform.
\end{oss}
The following result is due to Lin, see \cite[Lemma 3]{Lin}. For completeness, we provide a proof, slightly amended with respect to the original one. To the best of our knowledge, this is the best known results for general sets.
\begin{prop}[Simplicity for general sets]
\label{prop:lin}
Let $\Omega\subset\mathbb{R}^N$ be a $2-$admissible open connected set. Then there exists $2<q_0<2^*$ such that $\lambda_1(\Omega;q)$ is simple for every $2<q<q_0$.
\end{prop}
\begin{proof}
The proof exploits a contradiction argument. We assume that for every $2<q<2^*$, the first $q-$eigenvalue is not simple. Thus the problem 
\begin{equation}
\label{eq:2q}
\lambda_1(\Omega;q)=\min_{\varphi\in\mathcal{D}^{1,2}_0(\Omega)}\left\{\int_\Omega |\nabla \varphi|^2\,dx\, :\, \int_\Omega |\varphi|^q\,dx=1\right\},
\end{equation}
always admits (at least) two linearly independent solutions, which can be taken to be positive by Proposition \ref{lm:constantsign}. We take a decreasing sequence $\{q_n\}_{n\in\mathbb{N}}\subset (2,+\infty)$ such that 
\[
\lim_{n\to\infty} q_n=2.
\]
Correspondingly, there exist two distinct positive solutions of \eqref{eq:2q}. We call them $u_n$ and $v_n$, while denoting for simplicity
\[
\lambda_{q_n}=\lambda_1(\Omega;q_n).
\] 
We recall that (see for example \cite[Lemma 2.1]{BraBus})
\[
\lim_{n\to\infty} \int_\Omega |\nabla u_n|^2\,dx=\lim_{n\to\infty}\int_\Omega |\nabla v_n|^2\,dx=\lim_{n\to \infty} \lambda_{1}(\Omega;q_n)=\lambda_1(\Omega)=\int_\Omega |\nabla u|^2\,dx,
\]
where $u\in \mathcal{D}^{1,2}_0(\Omega)$ is the unique first positive eigenfunction of the Dirichlet-Laplacian, with unit $L^2(\Omega)$ norm. 
Then it is not difficult to see that
\begin{equation}
\label{convergenzaH1}
\lim_{n\to\infty} \|\nabla u_n-\nabla u\|_{L^2(\Omega)}=\lim_{n\to\infty} \|\nabla v_n-\nabla u\|_{L^2(\Omega)}=0.
\end{equation}
We also observe that, thanks to Proposition \ref{prop:Linfty}, we can assume 
\begin{equation}
\label{limitate}
\|u_n\|_{L^\infty(\Omega)}+\|v_n\|_{L^\infty(\Omega)}\le C,\qquad \mbox{ for every }n\in\mathbb{N}.
\end{equation}
This entails that for every $2<\gamma<+\infty$
\[
\|u_n-u\|_{L^\gamma(\Omega)}\le \|u_n-u\|_{L^\infty}^{1-\frac{2}{\gamma}}\,\|u_n-u\|_{L^2(\Omega)}^\frac{2}{\gamma}\le C\,\|u_n-u\|_{L^2(\Omega)}^\frac{2}{\gamma},
\]
and, similarly
\[
\|v_n-u\|_{L^\gamma(\Omega)}\le \|v_n-u\|_{L^\infty}^{1-\frac{2}{\gamma}}\,\|v_n-u\|_{L^2(\Omega)}^\frac{2}{\gamma}\le C\,\|v_n-u\|_{L^2(\Omega)}^\frac{2}{\gamma}.
\]
Thus from \eqref{convergenzaH1} and Poincar\'e inequality, we also get
\[
\lim_{n\to\infty} \|u_n-u\|_{L^\gamma(\Omega)}=\lim_{n\to\infty} \|v_n-u\|_{L^\gamma(\Omega)}=0,\qquad \mbox{ for every } 2\le \gamma<+\infty.
\]
By subtracting the two equations
\[
\int_\Omega \langle \nabla u_n,\nabla\varphi\rangle=\lambda_1(\Omega;q_n)\,\int_\Omega u_n^{q_n-1}\,\varphi\,dx,
\]
and
\[
\int_\Omega\langle \nabla v_n,\nabla\varphi\rangle=\lambda_1(\Omega;q_n)\,\int_\Omega v_n^{q_n-1}\,\varphi\,dx,
\]
we get
\[
\int_\Omega\langle \nabla (u_n-v_n),\nabla\varphi\rangle=\lambda_{q_n}\,\int_\Omega (u_n^{q_1-1}-v_n^{q_n-1})\,\varphi\,dx,\qquad \mbox{ for } \varphi\in\mathcal{D}^{1,2}_0(\Omega).
\]
We now observe that for every $a,b\ge 0$ we have
\begin{equation}
\label{funzioni}
\begin{split}
a^{q_n-1}-b^{q_n-1}&=\int_0^1 \frac{d}{dt} (t\,a+(1-t)\,b)^{q_n-1}\,dt\\
&=(q_n-1)\,\left(\int_0^1 (t\,a+(1-t)\,b)^{q_n-2}\,dt\right)\,(a-b).
\end{split}
\end{equation}
We thus get
\begin{equation}
\label{almostlinear}
\begin{split}
\int_\Omega \langle \nabla (u_n-v_n),\nabla \varphi\rangle\,dx=\lambda_{q_n}\,\int_\Omega w_n\,(u_n-v_n)\,\varphi\,dx\\
\end{split}
\end{equation}
where 
\[
w_n(x)=(q_n-1)\,\int_0^1 (t\,u_n(x)+(1-t)\,v_n(x))^{q_n-2}\,dt.
\]
For every $n\in\mathbb{N}$ we set
\[
\phi_n=\frac{u_n-v_n}{\|u_n-v_n\|_{L^2(\Omega)}}\in \mathcal{D}^{1,2}_0(\Omega),
\]
then from \eqref{almostlinear} we get that $\varphi_n$ solves the following weighted linear eigenvalue problem
\begin{equation}
\label{linearizedlin}
\int_\Omega \langle \nabla\phi_n,\nabla \varphi\rangle\,dx=\lambda_{q_n}\,\int_\Omega w_n\,\phi_n\,\varphi\,dx,\qquad \mbox{ for } \varphi\in \mathcal{D}^{1,2}_0(\Omega). 
\end{equation}
Observe that, since both $u_n$ and $v_n$ have unit $L^q(\Omega)$ norm, we can not have $u_n\ge v_n$ or $u_n\le v_n$ in $\Omega$. Thus we must have
\[
|\Omega_n^+|:=|\{x\in\Omega\, :\, u_n(x)>v_n(x)\}|>0\quad \mbox{ and }\quad |\Omega_n^-|:=|\{x\in\Omega\, :\, u_n(x)<v_n(x)\}|>0. 
\] 
This entails that the function $\phi_n$ must change sign.
We now claim that 
\begin{equation}
\label{conv2}
\|w_n\|_{L^\infty(\Omega)}\le C,\qquad \mbox{ for every }n\in\mathbb{N},
\end{equation}
and 
\begin{equation}
\label{conv3}
w_n \mbox{ converges (up to a subequence) in $L^2_{\rm loc}(\Omega)$ to $1$}.
\end{equation}
The first fact follows from \eqref{limitate}. To prove the second fact, we take $\Omega'\Subset\Omega$ and observe that
\[
\begin{split}
\int_{\Omega'} |w_n-1|^2\,dx&=\int_{\Omega'} \left|\int_0^1 \Big[(q_n-1)\, (t\,u_n(x)+(1-t)\,v_n(x))^{q_n-2}-1\Big]\,dt\right|^2\,dx\\
&\le \int_0^1 \int_{\Omega'} \Big|(q_n-1)\, (t\,u_n(x)+(1-t)\,v_n(x))^{q_n-2}-1\Big|^2\,dx\,dt\\
&\le 2\,(q_n-2)^2\,\int_0^1\,\int_{\Omega'} \Big||t\,u_n(x)+(1-t)\,v_n(x)|^{q_n-2}\Big|^2\,dx\,dt\\
&+2\,\int_0^1\int_{\Omega'} \Big||t\,u_n(x)+(1-t)\,v_n(x)|^{q_n-2}-1\Big|^2\,dx\,dt\\
&\le C\,(q_n-2)^2+2\,\int_0^1\int_{\Omega'} \Big||t\,u_n(x)+(1-t)\,v_n(x)|^{q_n-2}-1\Big|^2\,dx\,dt.
\end{split}
\]
We now observe that 
\[
\Big||t\,u_n(x)+(1-t)\,v_n(x)|^{q_n-2}-1\Big|^2\le C\cdot 1_{\Omega'}\in L^1(\Omega'\times[0,1]),
\]
still by \eqref{limitate}. In addition, by possibly passing to a subsequence, we have
\[
\lim_{n\to\infty}\Big||t\,u_n(x)+(1-t)\,v_n(x)|^{q_n-2}-1\Big|^2=0,\qquad \mbox{ for a.\,e. } (t,x)\in [0,1]\times\Omega'.
\]
Thus \eqref{conv3} now follows by using the Dominated Convergence Theorem.
\par
By choosing $\varphi=\phi_n$ in \eqref{linearizedlin} and using \eqref{conv2}, we get
\[
\int_\Omega |\nabla \phi_n|^2\,dx=\lambda_{q_n}\,\int_\Omega w_n\,|\phi_n|^2\,dx\le C\,\int_\Omega |\phi_n|^2\,dx=C.
\]
This shows that $\{\phi_n\}_{n\in\mathbb{N}}$ is bounded in $\mathcal{D}^{1,2}_0(\Omega)$.
Then there exists $\phi\in \mathcal{D}^{1,2}_0(\Omega)$ such that $\phi_n$ converges (up to a subsequence) to $\phi$, weakly in $\mathcal{D}^{1,2}_0(\Omega)$ and strongly in $L^2(\Omega)$ (thanks to the fact that $\Omega$ is $2-$admissible). In particular, we have
\[
\|\phi\|_{L^2(\Omega)}=1,
\]
thus the limit $\phi$ is not trivial. If we take $\varphi\in C^\infty_0(\Omega)$ and use \eqref{conv3}, we can now pass to the limit in \eqref{linearizedlin} and obtain that $\phi$ is a weak solution of 
\[
\int_\Omega \langle \nabla \phi,\nabla\varphi\rangle\,dx=\lambda_1(\Omega)\,\int_\Omega \phi\,\varphi\,dx. 
\]
By recalling that $\lambda_1(\Omega)$ is simple and that $\phi\in\mathcal{D}^{1,2}_0(\Omega)$ has unit $L^2(\Omega)$ norm, we must have $\phi=u$ or $\phi=-u$. In particular, $\phi$ has constant sign. 
\par
On the other hand, by recalling that $\phi_n$ is sign-changing, we can test \eqref{linearizedlin} with $\phi_n^+$ and $\phi_n^-$. This gives
\[
\int_\Omega |\nabla \phi_n^\pm|^2\,dx=\lambda_{q_n}\int_\Omega w_n\,|\phi_n^\pm|^2\,dx\le C\,\int_\Omega |\phi_n^\pm|^2\,dx,
\]
thanks to the uniform $L^\infty$ estimate \eqref{conv2} on $w_n$.
By using Poincar\'e inequality, we get 
\[
|\Omega_n^\pm|^\frac{2}{N}\,\int_\Omega |\nabla\phi_n^\pm|^2\,dx\ge \frac{1}{C'}\,\int_{\Omega}\, |\phi_n^\pm|^2\,dx,
\]
which in turn implies
\[
\frac{1}{C\,C'} \le |\Omega_n^\pm|^\frac{2}{N},\qquad \mbox{ for } n\in\mathbb{N}.
\]
This contradicts the fact that $\phi_n$ converges to the constant sign function $\phi$.
\end{proof}
An abstract sufficient condition in order to infer simplicity of $\lambda_1(\Omega;q)$ is contained in the following result, which is due to Damascelli, Grossi and Pacella. 
\begin{teo}[Non-degeneracy implies simplicity]
\label{teo:nondegen}
Let $\Omega\subset\mathbb{R}^N$ be an open bounded set such that for every $2<q<2^*$, the following condition is satisfied: 
\par\vskip.2cm
$\bullet$ for every first positive $q-$eigenfunction $U$ with unit $L^q$ norm, we have
\begin{equation}
\label{mu2}
\mu_2:=\min_{\mathcal{F}\in\Sigma_2(\Omega)}\left\{\max_{\varphi\in\mathcal{F}} \left(\int_\Omega |\nabla \varphi|^2\,dx-(q-1)\,\lambda_1(\Omega;q)\,\int_\Omega U^{q-2}\,\varphi^2\,dx\right)\right\}>0,
\end{equation}
where 
\[
\Sigma_2(\Omega)=\Big\{\mathcal{F}=F\cap\mathcal{S}_2(\Omega)\,:\, F \mbox{ $m$-dimensional subspace of } \mathcal{D}^{1,2}_0(\Omega)\, :\, m\ge 2\Big\}.
\]
Then $\lambda_1(\Omega;q)$ is simple for every $2<q<2^*$.
\end{teo}
\begin{proof}
It is not difficult to see that $\mu_2$ is the second eigenvalue of the linearized operator
\[
\varphi\mapsto -\Delta \varphi-(q-1)\,\lambda_{1}(\Omega;q)\, U^{q-2}\,\varphi.
\]
Then our assumption $\mu_2>0$ and Lemma \ref{lm:pesante} entail that $U$ is {\it non-degenerate}, i.e. $0$ is not an eigenvalue of such an operator. We thus conclude by applying \cite[Theorem 4.4]{DGP}.
\end{proof}
With the aid of the previous result, Proposition \ref{prop:lin} can be considerably improved for two dimensional convex sets. This is still due to Lin, see \cite[Theorem 1]{Lin}.
\begin{teo}[Simplicity for convex planar sets]
\label{teo:lin}
Let $\Omega\subset\mathbb{R}^2$ be an open bounded convex set. Then $\lambda_1(\Omega;q)$ is simple for every $2<q<2^*$.
\end{teo}
\begin{proof}
In view of Theorem \ref{teo:nondegen}, it is sufficient to prove that condition \eqref{mu2} is satisfied. By Lemma \ref{lm:pesante}, we already know that $0\le \mu_2$. Thus in order to conclude, we only need to show that $\mu_2\not=0$.
\par 
The proof of this fact is quite sophisticated, we reproduce Lin's argument contained in\footnote{The proof in \cite{Lin} tacitly assumes the boundary of $\Omega$ to be smooth. Here we avoid smoothness assumptions.} \cite[Lemma 2]{Lin}. We will indicate by $\mathcal{H}^{N-1}$ the $(N-1)-$dimensional Hausdorff measure.
\par
We argue by contradiction and assume that $\mu_2=0$. Thus there exists a nontrivial function $\phi$ such that
\begin{equation}
\label{eigensalonen}
-\Delta \phi-(q-1)\,\lambda_{1}(\Omega;q)\, U^{q-2}\,\phi=0,\quad \mbox{ in }\Omega,\qquad \phi=0,\quad \mbox{ on }\partial\Omega.
\end{equation}
We first observe that since $U\in L^\infty(\Omega)$ and $\Omega$ is convex,
we have that $\phi,U \in H^2(\Omega)$ by \cite[Theorem 3.2.1.2]{Gr}. This in turn implies that $\nabla\phi,\nabla U\in H^1(\Omega;\mathbb{R}^2)$ and thus they have a trace in $H^{1/2}(\partial\Omega)\hookrightarrow L^2(\partial\Omega)$.
Moreover, by Hopf's boundary Lemma, it holds
\begin{equation}
\label{hopfconvesso}
0>\frac{\partial U}{\partial \nu_\Omega}=-|\nabla U|,\qquad \mbox{ $\mathcal{H}^{N-1}-$a.\,e. on }\partial\Omega,
\end{equation}
where $\nu_\Omega$ is the exterior normal versor, which is well-defined $\mathcal{H}^{N-1}-$almost everywhere\footnote{It is sufficient to reproduce the standard proof of Hopf's boundary Lemma, by further using the following fact: {\it if $\Omega\subset\mathbb{R}^N$ is an open bounded convex set, then for $\mathcal{H}^{N-1}-$almost every $y_0\in\partial\Omega$ there exists $B_R(x_0)\subset\Omega$ such that
\[
\partial B_R(x_0)\cap\partial\Omega=\{y_0\}.
\]}
The proof of this ``almost everywhere internal ball condition'' can be achieved by using that $\partial\Omega$ is locally the graph of a convex function and that convex functions admits a second order Taylor expansion almost everywhere (the so-called {\it Alexandrov's Theorem}, see \cite[Chapter 6, Section 4, Theorem 1]{EG}).}. 
Then we define the new function
\begin{equation}
\label{w}
w(x)=\langle x-z_0,\nabla U(x)\rangle,\qquad x\in\Omega,
\end{equation}
where $z_0\in\mathbb{R}^2$ is a point that will be suitably chosen. Observe that $w\in W^{1,2}(\Omega)$.
By using the equation for $U$, the function $w$ weakly solves
\begin{equation}
\label{franz1}
-\Delta w-(q-1)\,\lambda_1(\Omega;q)\,U^{q-2}\,w=2\,U^{q-1}.
\end{equation}
By using the equations for $U$ and $\phi$, we get
\[
\begin{split}
0&=\int_\Omega \langle\nabla U,\nabla \phi\rangle\,dx-\int_\Omega \langle\nabla\phi,\nabla U\rangle\,dx\\
&=\lambda_1(\Omega;q)\,\int_\Omega U^{q-1}\,\phi\,dx-(q-1)\,\lambda_1(\Omega;q)\,\int_\Omega U^{q-1}\,\phi\,dx,
\end{split}
\]
which implies that 
\begin{equation}
\label{franz2}
\int_\Omega U^{q-1}\,\phi\,dx=0,
\end{equation}
thanks to the fact that $q\not =2$. By using \eqref{eigensalonen}, \eqref{franz1} and \eqref{franz2}, we  get\footnote{The boundary integral is well-defined, thanks to the fact that $w\in H^1(\Omega)$, $\nabla \phi\in H^1(\Omega;\mathbb{R}^2)$.}
\[
\begin{split}
\int_{\partial\Omega} w\,\frac{\partial\phi}{\partial\nu_\Omega}\,d\mathcal{H}^{N-1}&=\int_\Omega \langle \nabla w,\nabla \phi\rangle\,dx+\int_\Omega w\,\Delta\phi\,dx\\
&=(q-1)\,\lambda_1(\Omega;q)\,\int_\Omega U^{q-2}\,w\,\phi\,dx+2\,\int_\Omega U^{q-1}\,\phi\,dx\\
&-(q-1)\,\lambda_1(\Omega;q)\,\int_\Omega U^{q-2}\,w\,\phi\,dx,
\end{split}
\]
that is
\begin{equation}
\label{linoleum}
\int_{\partial\Omega} w\,\frac{\partial\phi}{\partial\nu_\Omega}\,d\mathcal{H}^{N-1}=0.
\end{equation}
The idea now is to exploit this identity and the convexity of $\Omega$, in order to contradict Hopf's boundary Lemma. 
\par
Since $\phi$ is a second eigenfunction of the linearized problem,
we can apply {\it Courant's Nodal Domains Theorem}\footnote{We recall that the proof of this result is based on the {\it Courant-Fischer-Weyl min-max formula} and the {\it unique continuation principle} for eigenfunctions. Both facts hold for the linearized operator 
\[
-\Delta-(q-1)\,\lambda_1(\Omega;q)\,U^{q-1},
\] 
thus one can easily adapt the classical proof of \cite[page 452]{CH}. For the unique continuation principle, we refer to \cite[Theorem II]{Mu}.}, to deduce that the nodal set 
\[
\overline{\{x\in\Omega\, :\, \phi(x)=0\}},
\]
divides $\Omega$ in exactly two sets. We thus have three cases:
\begin{itemize}
\item[(i)] the nodal set hits $\partial\Omega$ at one point;
\vskip.2cm
\item[(ii)] the nodal set hits $\partial\Omega$ at two points $x_0\not= x_1$ and there exist two directions $\omega_0,\omega_1\in\mathbb{S}^1$ such that
\[
L_i=\{x\in \mathbb{R}^2\, :\, \langle x-x_i,\omega_i\rangle= 0\},\qquad \mbox{ for } i=0,1, 
\]
are {\it supporting lines}\footnote{This means that 
\[
\overline{\Omega}\subset \{x\in \Omega\, :\, \langle x-x_i,\omega_i\rangle\le 0\},\qquad \mbox{ for } i=0,1.
\]} for $\Omega$, which are not parallel (see Figure \ref{fig:lin_potato});
\vskip.2cm
\item[(iii)] the nodal line set $\partial\Omega$ at two points $x_0\not= x_1$ and the supporting lines $L_0,L_1$ at these points are parallel (see Figure \ref{fig:lin_potato2}). 
\end{itemize}
\begin{figure}[h]
\includegraphics[scale=.3]{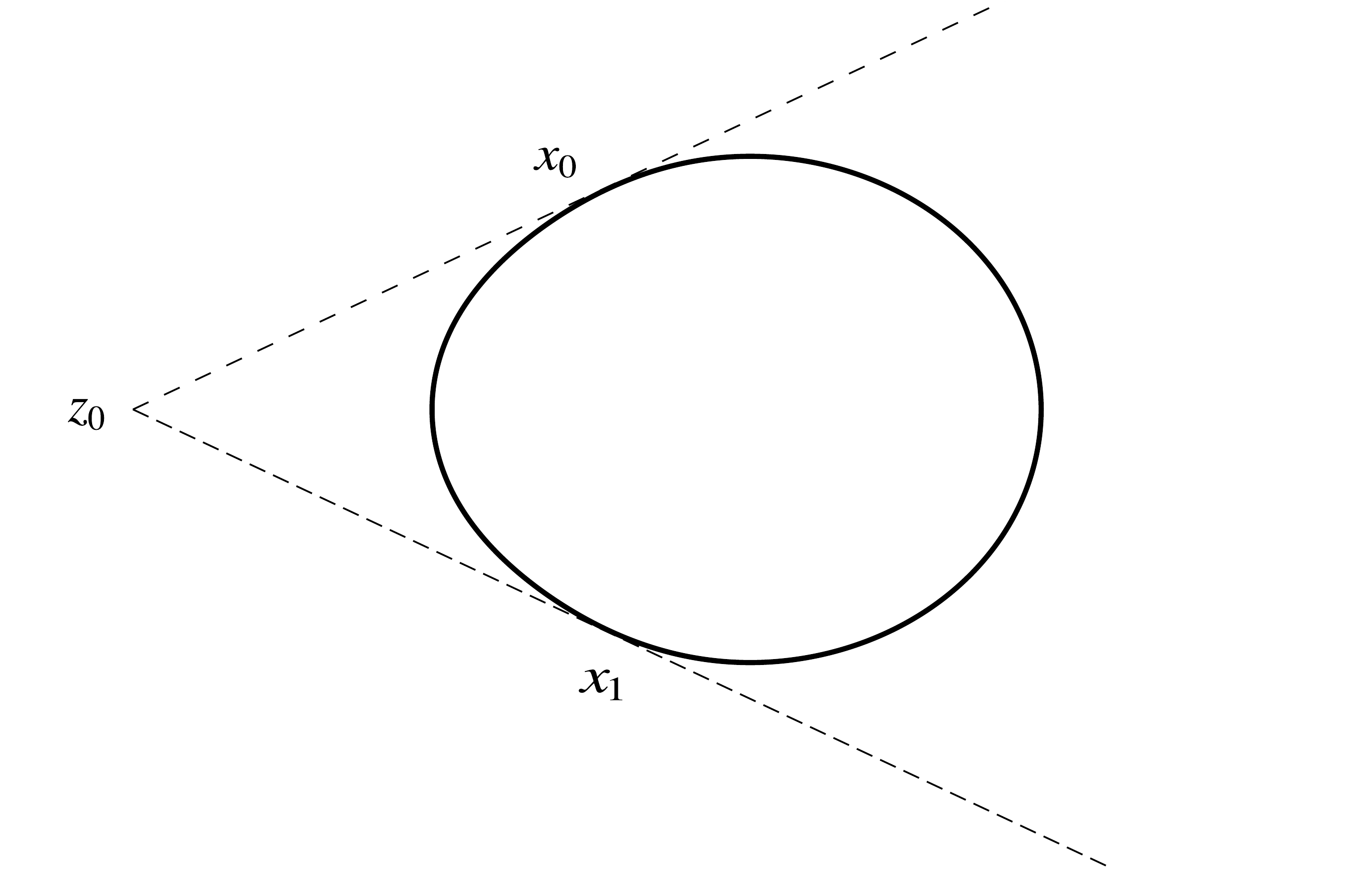}
\caption{Proof of Theorem \ref{teo:nondegen}, case (ii).}
\label{fig:lin_potato}
\end{figure}
\begin{figure}[h]
\includegraphics[scale=.3]{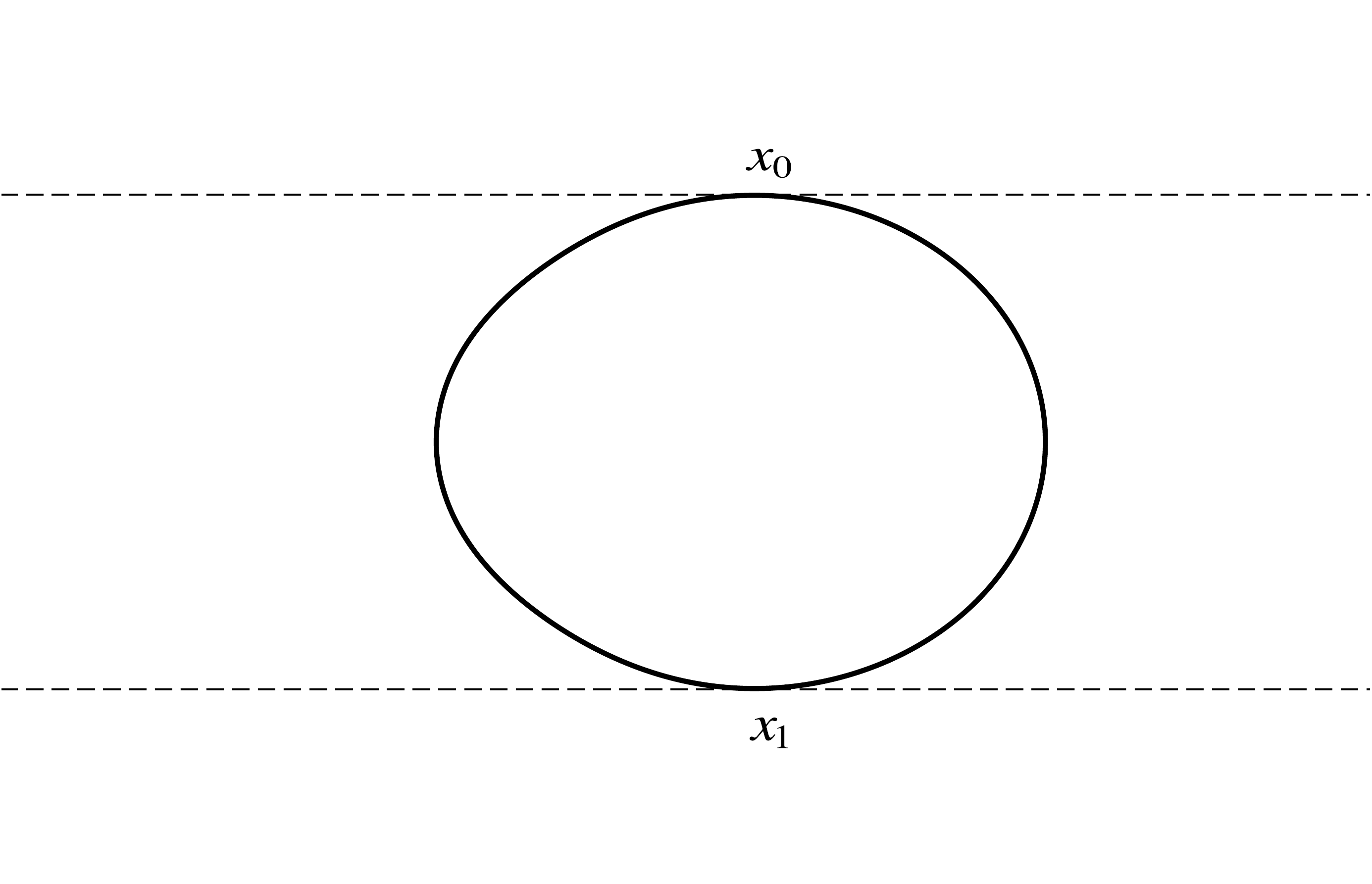}
\caption{Proof of Theorem \ref{teo:nondegen}, case (iii).}
\label{fig:lin_potato2}
\end{figure}
Case (i) is the simplest one: by taking $z_0$ to be any interior point of $\Omega$, by convexity we have
\[
w(x)=\langle x-z_0,\nabla U(x)\rangle<0,\qquad \mbox{ for $\mathcal{H}^{N-1}-$a.\,e. } x\in\partial\Omega.
\]
Here we used \eqref{hopfconvesso}. Moreover, the normal derivative $\partial\phi/\partial\nu_\Omega$ must have constant sign on $\partial\Omega$. The last two informations, inserted in \eqref{linoleum}, entail that 
\[
\frac{\partial \phi}{\partial\nu_\Omega}=0,\qquad \mbox{ $\mathcal{H}^{N-1}-$a.\,e. on } \partial\Omega.
\]
This contradict Hopf's boundary Lemma.
\par
In case (ii), we choose $z_0\in\mathbb{R}^2\setminus\Omega$ to be the intersection of the two supporting lines $L_0$ and $L_1$, see Figure \ref{fig:lin_potato}. Observe that it may happen that $z_0\in\partial\Omega$. We now divide $\partial\Omega$ has follows: $E_-$ is the curve on $\partial\Omega$ connecting $x_0$ to $x_1$, in counter-clockwise sense; then $E_+=\partial\Omega\setminus E_-$. By construction, we have
\begin{equation}
\label{pop}
\int_{\partial\Omega} w\,\frac{\partial\phi}{\partial\nu_\Omega}\,d\mathcal{H}^{N-1}=\int_{E_+} w\,\frac{\partial\phi}{\partial\nu_\Omega}\,d\mathcal{H}^{N-1}+\int_{E_-} w\,\frac{\partial\phi}{\partial\nu_\Omega}\,d\mathcal{H}^{N-1},
\end{equation}
and moreover, thanks to convexity, we have $w<0$ on $E_-$ and $w>0$ on $E_+$. The function $\phi$ has constant sign on the domain enclosed by $E_-$ and the nodal set, assume for simplicity that we have $\phi>0$. Then by Hopf's boundary Lemma
\[
\frac{\partial\phi}{\partial\nu_\Omega}<0,\qquad \mbox{ for $\mathcal{H}^{N-1}-$a.\,e. } x\in E_-.
\] 
Similary, by using that $\phi<0$ on the domain enclosed by $E_+$ and the nodal set, we get
\[
\frac{\partial\phi}{\partial\nu_\Omega}>0,\qquad \mbox{ for $\mathcal{H}^{N-1}-$a.\,e. } x\in E_+.
\] 
By using these sign informations in \eqref{pop}, we get
\[
\int_{\partial\Omega} w\,\frac{\partial\phi}{\partial\nu_\Omega}\,d\mathcal{H}^{N-1}>0,
\]
which contradicts \eqref{linoleum}.
\par
Finally, in case (iii), by assuming for simplicity that $L_0$ and $L_1$ are parallel to the $x_1$ axis, we change the choice \eqref{w} of $w$ and replace it with the following one
\[
w=\frac{\partial U}{\partial x_1}.
\]
It is not difficult to see that \eqref{linoleum} still holds\footnote{In this part, the paper \cite{Lin} contains a misprint. The term $\partial\phi/\partial x_1$ there must be replaced by $\partial\phi/\partial\nu_\Omega$.}. Then one can proceed as in case (ii) and get the conclusion in this case, as well.
\end{proof}
\begin{oss}
\label{oss:dancer}
A result analogous to Theorem \ref{teo:lin} was previously proved by Dancer for smooth bounded planar sets $\Omega\subset\mathbb{R}^2$ such that (see Figure \ref{fig:dancer}): 
\begin{itemize}
\item $\Omega$ is convex in the directions $x_1$ and $x_2$;
\vskip.2cm
\item $\Omega$ is symmetric with respect to the hyperplanes $\{x_1=0\}$ and $\{x_2=0\}$,
\end{itemize}
see \cite[Theorem 5]{Da1988}. Later on, Dancer's result was obtained again by Damascelli, Grossi and Pacella in \cite[Theorem 4.1]{DGP}, by using a different proof based on minimum principles.
\begin{figure}
\includegraphics[scale=.4]{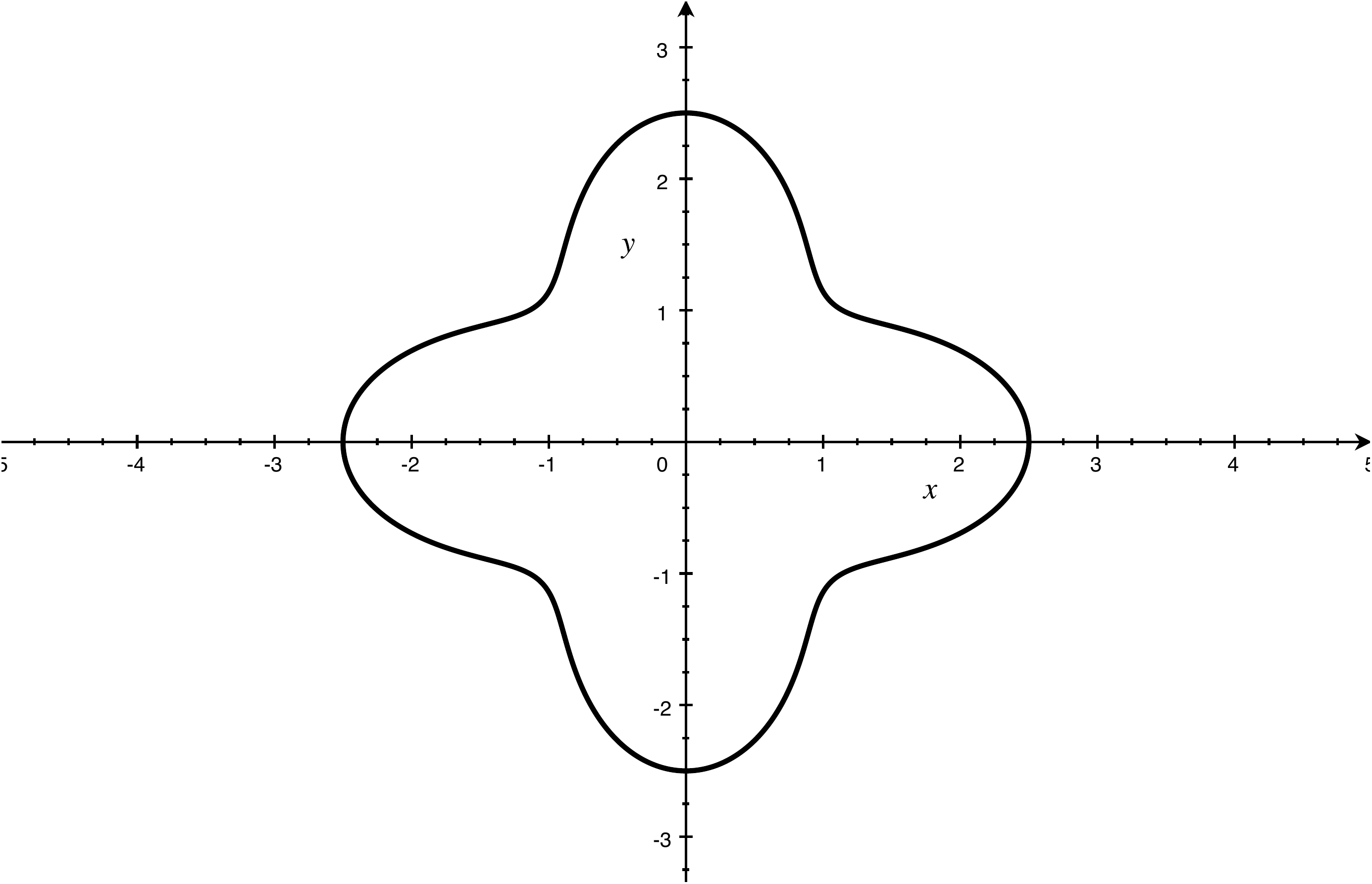}
\caption{A non-convex set verifying the assumption of Remark \ref{oss:dancer}.}
\label{fig:dancer}
\end{figure}
\end{oss}

\subsection{Counter-examples}
A well-known counter-example due to Nazarov shows that for $q>2$:
 \begin{enumerate}
\item $\lambda_1(\Omega;q)$ may not be simple;
\vskip.2cm
\item there may exist a $q-$eigenvalue $\lambda>\lambda_1(\Omega;q)$ with positive eigenfunctions.
\end{enumerate}
The set $\Omega$ considered by Nazarov is a {\it spherical shell}, i.e. a set with nontrivial topology, see \cite[Proposition 1.2]{Na}.
 \vskip.2cm
The following example shows that the same phenomena can appear even if the set has a trivial topology. Indeed, observe that the sets $\Omega_\varepsilon$ below are contractible. More precisely, they are starshaped. This shows that the simplicity of $\lambda_1(\Omega;q)$ for $2<q<2^*$ is linked to the geometry of the underlying set $\Omega$ and not simply to its topology.
\begin{figure}
\includegraphics[scale=.4]{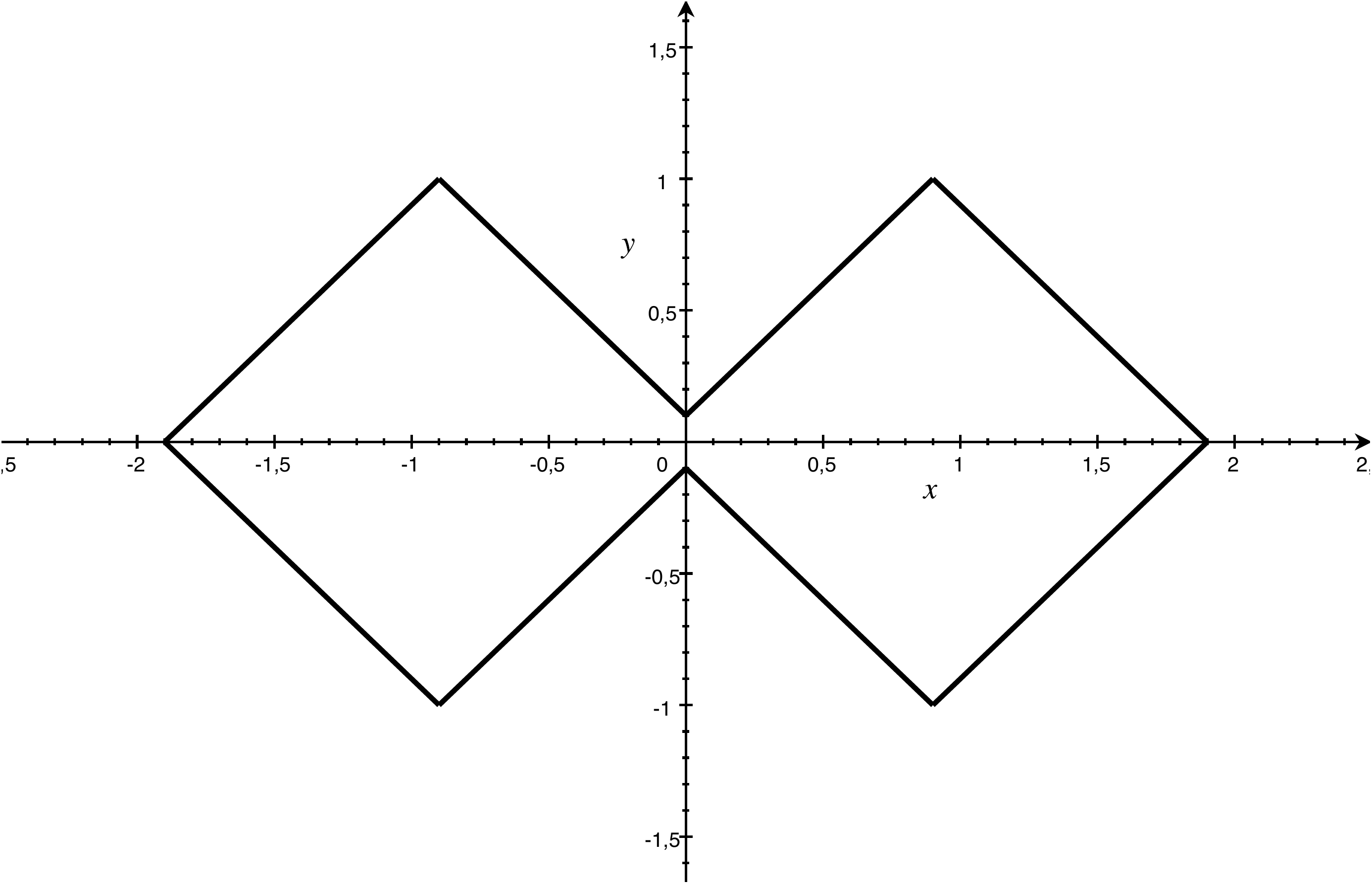}
\caption{The set $\Omega_\varepsilon$ of Example \ref{exa:dancer}.}
\end{figure}
\begin{exa}
\label{exa:dancer}
Let $2<q<2^*$ and $0<\varepsilon<1$, we indicate by $Q_1$ the cube
\[
Q_1=\left\{x\in\mathbb{R}^N\, :\, \sum_{i=1}^N |x_i|<1\right\}.
\]
We set
\[
\Omega^+_\varepsilon=\Big(Q_1+(1-\varepsilon)\,\mathbf{e}_1\Big)\cap\Big\{(x_1,x')\in\mathbb{R}^N\, :\, x_1\ge 0\Big\},
\]
and
\[
\Omega^-_\varepsilon=\Big(Q_1-(1-\varepsilon)\,\mathbf{e}_1\Big)\cap\Big\{(x_1,x')\in\mathbb{R}^N\, :\, x_1\le 0\Big\}.
\]
Then we consider the open set
\[
\Omega_\varepsilon=\Omega^+_\varepsilon\cup\Omega^-_\varepsilon,
\]
consisting of two overlapping cubes centered at $(-1+\varepsilon)\,\mathbf{e}_1$ and $(1-\varepsilon)\,\mathbf{e}_1$, both having side $1$. There exists $\varepsilon_0=\varepsilon_0(N,q)>0$ such that for every $0<\varepsilon\le \varepsilon_0$, we have:
\begin{enumerate}
\item $\lambda_1(\Omega_\varepsilon;q)$ is not simple;
\vskip.2cm
\item there exists a $q-$eigenvalue $\lambda>\lambda_1(\Omega_\varepsilon;q)$ with positive eigenfunctions.
\end{enumerate}
\end{exa}
\begin{proof}
It is sufficient to prove that, at least for $\varepsilon>0$ small enough, any first $q-$eigenfunction $U_\varepsilon$ does not inherit the symmetry about the hyperplane $x_1=0$ from the set $\Omega_\varepsilon$. Indeed, if this were the case, then the two functions
\[
U_\varepsilon(x_1,x')\qquad \mbox{ and }\qquad U_\varepsilon(-x_1,x'),
\]
would give a pair of linearly independent first $q-$eigenfunctions.
In other words, we just need to prove that for $\varepsilon\ll 1$, we have
\begin{equation}
\label{dancer}
\lambda_1^{\rm sym}(\Omega_\varepsilon;q):=\min_{u\in \mathcal{D}^{1,2}_0(\Omega_\varepsilon)\setminus\{0\}} \left\{\frac{\displaystyle\int_{\Omega_\varepsilon} |\nabla u|^2\,dx}{\left(\displaystyle\int_{\Omega_\varepsilon} |u|^q\,dx\right)^\frac{2}{q}}\, :\, u(x_1,x')=u(-x_1,x')\right\}>\lambda_1(\Omega_\varepsilon;q).
\end{equation}
It is easy to see that the quantity $\lambda_1^{\rm sym}(\Omega_\varepsilon;q)$ defines a $q-$eigenvalue for $\Omega_\varepsilon$. We take $u_\varepsilon$ optimal for the variational problem which defines $\lambda_1^{\rm sym}(\Omega_\varepsilon;q)$. Without loss of generality, we can assume that $u_\varepsilon\ge 0$ and that
\[
\int_{\Omega^+_\varepsilon} |u_\varepsilon|^q\,dx=\int_{\Omega^-_\varepsilon} |u_\varepsilon|^q\,dx=1.
\]
In order to prove \eqref{dancer}, we first observe that by symmetry
\[
\begin{split}
\lambda_1^{\rm sym}(\Omega_\varepsilon;q)&=\frac{\displaystyle2\,\int_{\Omega^+_\varepsilon} |\nabla u_\varepsilon|^2\,dx}{\left(\displaystyle2\,\int_{\Omega^+_\varepsilon} |u_\varepsilon|^q\,dx\right)^\frac{2}{q}}\\
&=\min_{u\in H^1(\Omega^+_\varepsilon)\setminus\{0\}} \left\{\frac{\displaystyle2\,\int_{\Omega^+_\varepsilon} |\nabla u|^2\,dx}{\left(\displaystyle2\,\int_{\Omega^+_\varepsilon} |u|^q\,dx\right)^\frac{2}{q}}\, :\, u=0 \mbox{ on }\partial\Omega^+_\varepsilon\cap \{x_1>0\}\right\}\\
&=:2^{1-\frac{2}{q}}\,\mu_q(\Omega^+_\varepsilon).
\end{split}
\]
On the other hand, by using that $Q_1+(1-\varepsilon)\,\mathbf{e}_1\subset \Omega_\varepsilon$, by \eqref{mono} we immediately get
\begin{equation}
\label{ciaociao}
\lambda_1(\Omega_\varepsilon;q)\le \lambda_1(Q_1;q).
\end{equation}
We now claim that 
\begin{equation}
\label{limite}
\lim_{\varepsilon\to 0^+}\mu_q(\Omega^+_\varepsilon)=\lambda_1(Q_1;q).
\end{equation}
Observe that once we prove \eqref{limite}, the claimed estimate \eqref{dancer} easily follows from \eqref{ciaociao}, since 
\[
\lambda_1^{\rm sym}(\Omega_\varepsilon;q)=2^{1-\frac{2}{q}}\,\mu_q(\Omega^+_\varepsilon),
\]
and the factor $2^{1-2/q}$ is strictly larger than $1$, thanks to the fact that $q>2$.  
\par
In order to prove \eqref{limite}, we first observe that the first $q-$eigenfunction of the rescaled cube 
\[
\Big((1-\varepsilon)\,Q_1\Big)+ (1-\varepsilon)\,\mathbf{e}_1,
\] 
is admissible in the variational problem which defines $\mu_q(\Omega^+_\varepsilon)$. This entails
\begin{equation}
\label{upper bound}
\mu_q(\Omega^+_\varepsilon)\le \lambda_1((1-\varepsilon)\,Q_1;q)=(1-\varepsilon)^{N-2-\frac{2}{q}\,N}\,\lambda_1(Q_1;q).
\end{equation}
In turn, we immediately get
\[
\limsup_{\varepsilon\to 0^+}\mu_q(\Omega^+_\varepsilon)\le \lambda_1(Q_1;q).
\]
We have to show that 
\[
\liminf_{\varepsilon\to 0^+}\mu_q(\Omega^+_\varepsilon)\ge \lambda_1(Q_1;q).
\]
We take $u_\varepsilon\in \mathcal{D}^{1,2}_0(\Omega_\varepsilon)$ to be optimal for the variational problem which defines $\lambda_1^{\rm sym}(\Omega_\varepsilon;q)$ and consider its restriction to $\Omega_\varepsilon^+$. For simplicity, we can assume that $u_\varepsilon$ has unit $L^q$ norm on $\Omega^+_\varepsilon$. We also take
\[
\eta_\varepsilon(x_1)=\eta\left(\frac{x_1}{\varepsilon}\right),
\]
where $\eta\in C^\infty_0(\mathbb{R})$ is a non-negative and non-decreasing function, such that $\eta(t)=1$ for $t\ge 2$ and $\eta(0)=0$ for $t\le 1$. We then use the test function $\eta_\varepsilon\,u_\varepsilon$, so to get
\[
\begin{split}
\lambda_1(Q_1;q)&=\lambda_1(Q_1+(1-\varepsilon)\,\mathbf{e}_1;q)\\
&\le \frac{\displaystyle\int_{Q_1+(1-\varepsilon)\,\mathbf{e}_1}\Big[|\eta'_\varepsilon|^2\,|u_\varepsilon|^2+|\nabla u_\varepsilon|^2\,|\eta_\varepsilon|^2+2\,\partial_{x_1} u_\varepsilon\,\eta'_\varepsilon\,u\,\eta_\varepsilon\Big]\,dx}{\left(\displaystyle \int_{Q_1+(1-\varepsilon)\,\mathbf{e}_1}|\eta_\varepsilon|^q\,|u_\varepsilon|^q\,dx\right)^\frac{2}{q}}.
\end{split}
\]
We start by estimating the denominator. We have
\[
\begin{split}
\int_{Q_1+(1-\varepsilon)\,\mathbf{e}_1}|\eta_\varepsilon|^q\,|u_\varepsilon|^q\,dx&=\int_{\Omega_\varepsilon^+}|\eta_\varepsilon|^q\,|u_\varepsilon|^q\,dx\\
&=1-\int_{\{x\in Q_1+(1-\varepsilon)\,\mathbf{e}_1\, :\, 0<x_1<2\,\varepsilon\}}(1-|\eta_\varepsilon|^q)\,|u_\varepsilon|^q\,dx\\
&\ge 1- \|u_\varepsilon\|^q_{L^\infty(\Omega_\varepsilon^+)}\,|\{x\in Q_1+(1-\varepsilon)\,\mathbf{e}_1\, :\, 0<x_1<2\,\varepsilon\}|\\
&\ge 1-C\,\|u_\varepsilon\|^q_{L^\infty(\Omega_\varepsilon^+)}\,\varepsilon^N\\
&\ge 1-\widetilde C\,\varepsilon^N,
\end{split}
\]
where we used Proposition \ref{prop:Linfty} to bound uniformly the $L^\infty$ norm of the $q-$eigenfunction $u_\varepsilon$. By taking $\varepsilon>0$ small enough, raising to the power $-2/q$ and using the elementary inequality
\[
(1-t)^{-\frac{2}{q}}\le 1+2\,(2^{2/q}-1)\,t,\qquad \mbox{ for } 0\le t\le \frac{1}{2},
\]
we get 
\[
\left(\displaystyle \int_{Q_1+(1-\varepsilon)\,\mathbf{e}_1}|\eta_\varepsilon|^q\,|u_\varepsilon|^q\,dx\right)^{-\frac{2}{q}}\le 1+C\,\varepsilon^N,
\]
for a constant $C>0$ independent of $\varepsilon$.
Up to now, we obtained
\begin{equation}
\label{meuamigo}
\lambda_1(Q_1;q)\le \left(1+C\,\varepsilon^N\right)\int_{Q_1+(1-\varepsilon)\,\mathbf{e}_1}\Big[|\nabla u_\varepsilon|^2\,|\eta_\varepsilon|^2+|\eta'_\varepsilon|^2\,|u_\varepsilon|^2+2\,\partial_{x_1} u_\varepsilon\,\eta'_\varepsilon\,u_\varepsilon\,\eta_\varepsilon\Big]\,dx.
\end{equation}
In order to estimate the last integral, we separately estimate each integrand as follows: for the first one, we simply have
\[
\int_{Q_1+(1-\varepsilon)\,\mathbf{e}_1} |\nabla u_\varepsilon|^2\,|\eta_\varepsilon|^2\,dx\le \int_{\Omega^+_\varepsilon} |\nabla u_\varepsilon|^2\,dx=\mu_q(\Omega^+_\varepsilon).
\]
For the second integral, we observe that
\begin{equation}
\label{maledetto}
\begin{split}
\int_{Q_1+(1-\varepsilon)\,\mathbf{e}_1}&|\eta'_\varepsilon|^2\,|u_\varepsilon|^2\,dx\\
&\le \frac{1}{\varepsilon^2}\,\|\eta'\|_{L^\infty}^2\,\|u_\varepsilon\|^2_{L^\infty(\Omega_\varepsilon^+)}\,|Q_1+(1-\varepsilon)\,\mathbf{e}_1\cap \{\varepsilon<x_1<2\,\varepsilon\}|\\
&\le \frac{C}{\varepsilon^2}\,\varepsilon^N=C\,\varepsilon^{N-2}.
\end{split}
\end{equation}
The third integral can be treated similarly, by observing that
\[
\begin{split}
2\,\int_{Q_1+(1-\varepsilon)\,\mathbf{e}_1}\partial_{x_1} u_\varepsilon\,\eta'_\varepsilon\,u_\varepsilon\,\eta_\varepsilon\,dx&\le 2\, \left(\int_{\{\varepsilon<x_1<2\,\varepsilon\}}|\nabla u_\varepsilon|^2\,\eta_\varepsilon^2\,dx\right)^\frac{1}{2}\,\left(\int |\eta'_\varepsilon|^2\,u_\varepsilon^2\,dx\right)^\frac{1}{2}\\
&\le 2\,\sqrt{\mu_q(\Omega_\varepsilon^+)}\,\left(\int |\eta'_\varepsilon|^2\,u_\varepsilon^2\,dx\right)^\frac{1}{2}.
\end{split}
\]
By recalling the uniform estimate \eqref{upper bound}, from \eqref{meuamigo} we finally get for every $0<\varepsilon\le \varepsilon_0$
\[
\lambda_1(Q_1;q)\le \left(1+C\,\varepsilon^N\right)\,\left(\mu_q(\Omega^+_\varepsilon)+C_1\,\varepsilon^{N-2}+C_2\,\varepsilon^\frac{N-2}{2}\right),
\]
for some constants $C,C_1,C_2$ independent of $\varepsilon$.
This estimate is sufficient to conclude in the case $N\ge 3$. Indeed, in this case we get
\[
\lambda_1(Q_1;q)\le \liminf_{\varepsilon\to 0} \left(1+C\,\varepsilon^\frac{N+1}{2}\right)\,\left(\mu_q(\Omega^+_\varepsilon)+C_1\,\varepsilon^{N-2}+C_2\,\varepsilon^\frac{N-2}{2}\right)=\liminf_{\varepsilon\to 0} \mu_q(\Omega^+_\varepsilon),
\]
and this, in turn, concludes the proof of \eqref{limite}. 
\par
The case $N=2$ is slightly more complicate, in this case the estimate \eqref{maledetto} is a bit too rough. We need a more precise H\"older--type estimate of $u_\varepsilon$ near the junction part between $\Omega^+_\varepsilon$ and $\Omega^-_\varepsilon$. We proceed like this: we take polar coordinates $(\varrho,\vartheta)$ centered at $(0,-\varepsilon)$. Here $\varrho$ stands for the distance from the ``center'' $(0,-\varepsilon)$ and $\vartheta$ is the angle measuring the deviation from the semiaxis of negative $x_2$. Then we consider the barrier function
\[
\psi(\varrho,\vartheta)=C\,\left[\sqrt{\frac{\varrho}{2}}\,\sin\left(\frac{\vartheta}{2}\right)-\frac{\varrho}{2}\,\sin\left(\frac{\pi}{8}\right)\right],
\]
see Figure \ref{fig:barriera_quadrata}.
Observe that by construction we have
\[
\psi\ge 0 \mbox{ on } \partial\Omega_\varepsilon\qquad \mbox{ and }\qquad -\Delta \psi=C\,\Delta\left(\frac{\varrho}{2}\,\sin\left(\frac{\pi}{8}\right)\right)=\frac{C}{2}\,\sin\left(\frac{\pi}{8}\right)\,\frac{1}{\varrho}.
\]
Thus, up to choose $C>0$ large enough (uniformly in $\varepsilon$), we get
\[
-\Delta \psi\ge -\Delta u_\varepsilon,\qquad \mbox{ in }\Omega_\varepsilon.
\]
By the comparison principle, we obtain 
\[
0\le u_\varepsilon\le \psi\qquad \mbox{ in } \Omega_\varepsilon,
\]
which in turn implies that
\[
\begin{split}
\int_{Q_1+(1-\varepsilon)\,\mathbf{e}_1} |\eta'_\varepsilon|^2\,|u_\varepsilon|^2\,dx&\le \int_{Q_1+(1-\varepsilon)\,\mathbf{e}_1} |\eta'_\varepsilon|^2\,|\psi|^2\,dx\\
&\le\frac{\|\eta'\|_{L^\infty}^2}{\varepsilon^2}\,\int_{Q_1+(1-\varepsilon)\,\mathbf{e}_1\cap \{\varepsilon<x_1<2\,\varepsilon\}} |\psi|^2\,dx.
\end{split}
\]
It is only left to observe that $|\psi|\le C\,\sqrt{\varepsilon}$ on the set $Q_1+(1-\varepsilon)\,\mathbf{e}_1\cap \{\varepsilon<x_1<2\,\varepsilon\}$. This is now sufficient to conclude the proof as in the case $N\ge 3$.
\begin{figure}
\centering
\begin{minipage}{.49\textwidth}\centering
\includegraphics[width=7.5cm]{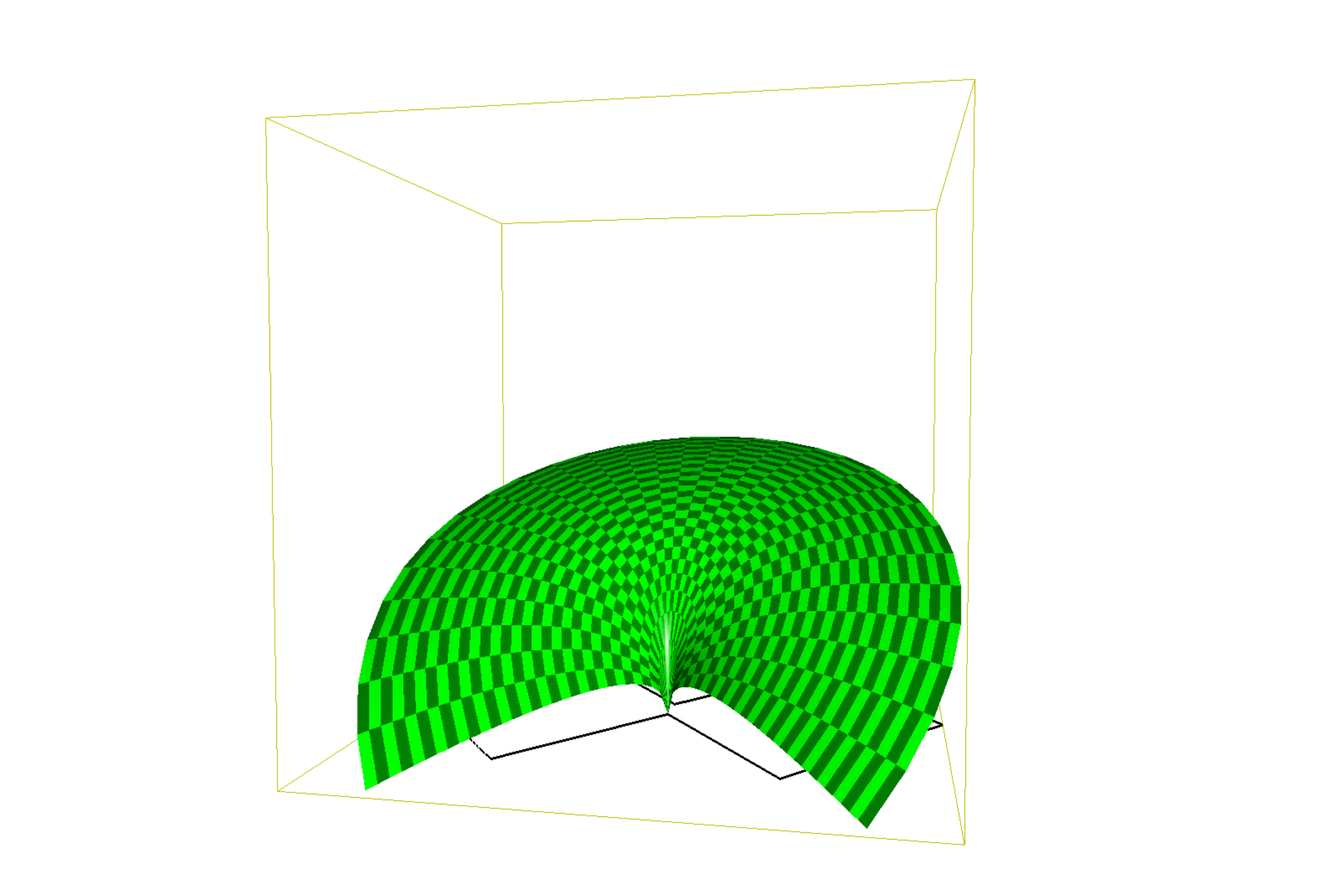}
\end{minipage}
\begin{minipage}{.49\textwidth}\centering
\includegraphics[width=7.5cm]{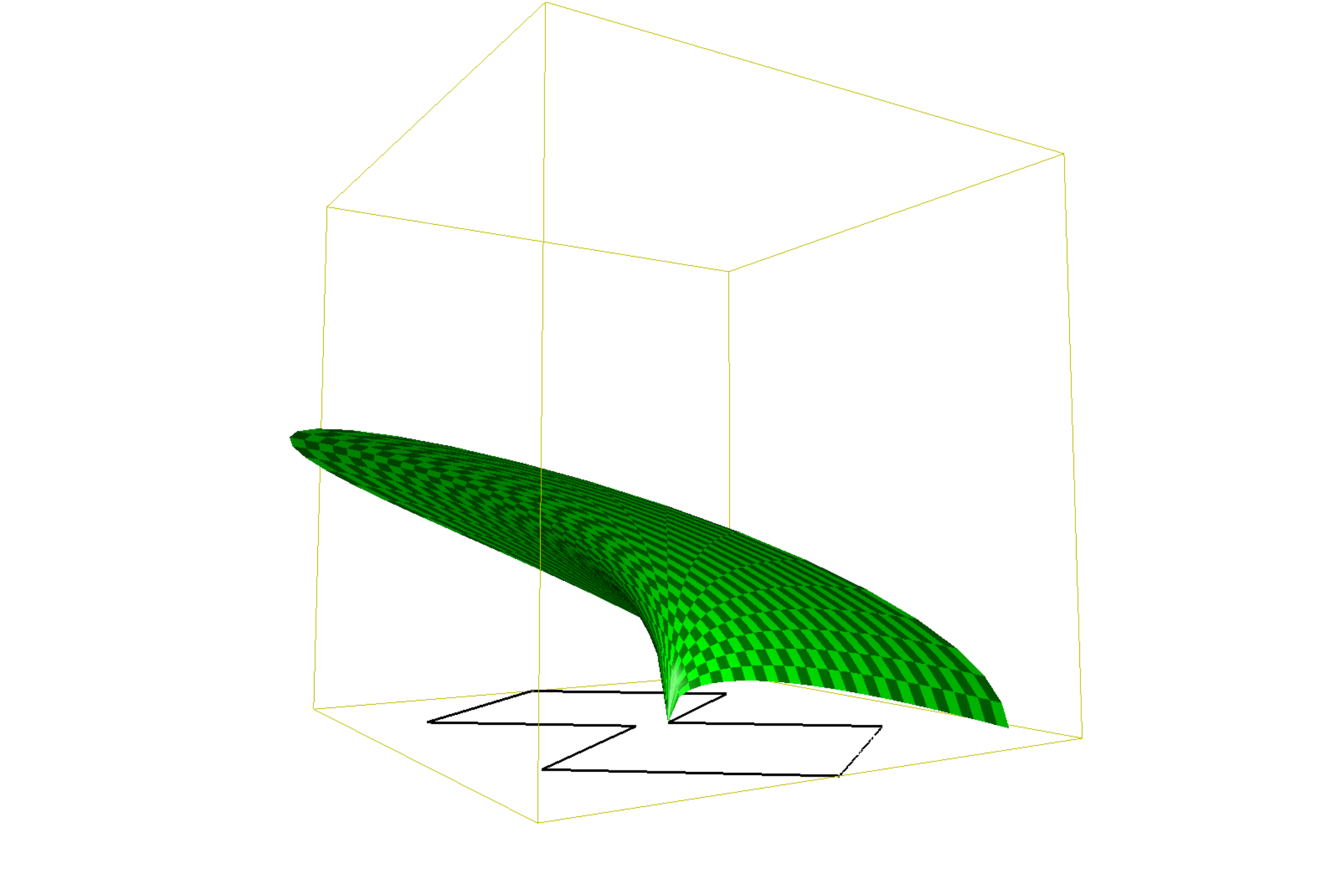}
\end{minipage}
\caption{The graph of the barrier function $\psi$, neeeded to handle Example \ref{exa:dancer} in the case $N=2$. In black, the boundary of the set $\Omega_\varepsilon$.}
\label{fig:barriera_quadrata}
\end{figure}
\end{proof}

\begin{oss}
The previous example is inspired by an inspection of the papers \cite{Da1988} and \cite{Da}. 
\end{oss}

\subsection{Open problems} We list here some questions for the case $2<q<2^*$ which, to the best of our knowledge, are open.

\begin{open}
On a ``good'' open set $\Omega\subset\mathbb{R}^N$, the $q-$spectrum is discrete and 
\[
\mathrm{Spec}(\Omega;q)=\mathrm{Spec}_{LS}(\Omega;q).
\]
\end{open}

\begin{open}
The first $q-$eigenvalue is isolated. 
\end{open}
\begin{oss}
We point out that, as observed in \cite{Er}, the isolation of $\lambda_1(\Omega;q)$ holds true whenever this is simple. However, it may happen that the first $q-$eigenvalue is isolated, even when this is not simple.
\end{oss}

\begin{open}
If $\Omega$ is connected, there exists only a finite number of $q-$eigenvalues with constant sign eigenfunctions. 
\end{open}

\begin{open}
Lin's Theorem \ref{teo:lin} is valid for open bounded convex sets in any dimension $N\ge 2$.
\end{open}
\appendix

\section{Spectrum of the linearized operator}

The next result can be found in Lin's paper \cite{Lin}, see Lemma 1 there.
\begin{lm}
\label{lm:pesante}
Let $2<q<2^*$ and let $\Omega\subset\mathbb{R}^N$ be a $q-$admissible open set. Let $U\in \mathcal{D}^{1,2}_0(\Omega)$ be a first positive $q-$eigenfunction, with unit $L^q$ norm. We consider the spectrum $\{\mu_1,\mu_2,\dots\}$ of the linearized operator
\begin{equation}
\label{linearized}
\varphi\mapsto -\Delta \varphi-(q-1)\,\lambda_{1}(\Omega;q)\, U^{q-2}\,\varphi,
\end{equation}
with homogeneous Dirichlet boundary conditions on $\partial\Omega$. Then
\[
\mu_1<0\le \mu_2\le \mu_3\le \dots.
\]
\end{lm}
\begin{proof}
By Proposition \ref{prop:Linfty}, the potential $U^{q-2}$ is bounded. Moreover, the embedding $\mathcal{D}^{1,2}_0(\Omega)\hookrightarrow L^2(\Omega)$ is compact by assumption (recall Remark \ref{oss:qadmissible}). This implies that the 
resolvent operator\footnote{We recall that this is the operator $\mathcal{R}:L^2(\Omega)\to L^2(\Omega)$ such that $\mathcal{R}(\varphi)$ is the unique solution in $\mathcal{D}^{1,2}_0(\Omega)$ of $-\Delta u+V\,u=\varphi$.} of 
\[
\varphi\mapsto -\Delta \varphi+V\,\varphi,\qquad \mbox{ with } V=(q-1)\,\lambda_{1}(\Omega;q)\, (\|U\|_{L^\infty(\Omega)}^{q-2}-U^{q-2})\ge0,
\]
is compact, positive and self-adjoint. By applying the Spectral Theorem,
we grant the existence of an infinite sequence of eigenvalues diverging to $+\infty$ for the last operator. We call them $0<\eta_1\le \eta_2\le \dots\le \eta_k\le \dots\nearrow +\infty$ and notice that for them we still have the Courant-Fischer-Weyl min-max principle.
\par
If we now set 
\[
\mu_k=\eta_k-(q-1)\,\lambda_1(\Omega;q)\,\|U\|_{L^\infty(\Omega)},\qquad \mbox{ for every } k\in\mathbb{N},
\]
we get the spectrum of \eqref{linearized}.
It is not difficult to see that the first eigenvalue
\[
\mu_1:=\inf_{\varphi\in \mathcal{D}^{1,2}_0(\Omega)} \left\{\int_\Omega |\nabla \varphi|^2\,dx-(q-1)\,\lambda_{1}(\Omega;q)\, \int_\Omega U^{q-2}\,|\varphi|^2\,dx\, :\, \int_\Omega |\varphi|^2\,dx=1\right\},
\] 
is strictly negative. It is sufficient to use the test function $\varphi=U/\|U\|_{L^2(\Omega)}$, so to get
\[
\mu_1\le \frac{\displaystyle\int_\Omega |\nabla U|^2\,dx-(q-1)\,\lambda_1(\Omega;q)\,\int_\Omega |U|^q\,dx}{\displaystyle\int_\Omega |U|^2\,dx}=\frac{\displaystyle (2-q)\,\lambda_1(\Omega;q)}{\displaystyle\int_\Omega |U|^2\,dx}<0.
\]
For the second eigenvalue $\mu_2$, we first observe that the minimality of $U$ entails that the function
\[
f(t)=\frac{\displaystyle\int_\Omega |\nabla U+t\,\nabla\varphi|^2\,dx}{\displaystyle\left(\int_\Omega |U+t\,\varphi|^q\,dx\right)^\frac{2}{q}}
\]
is minimal at $t=0$, for every $\varphi\in \mathcal{D}^{1,2}_0(\Omega)$. We thus must have $f''(t)\ge 0$, which implies after a routine computation
\begin{equation}
\label{egolo}
\int_\Omega |\nabla \varphi|^2-(q-1)\,\lambda_1(\Omega;q)\,\int_\Omega U^{q-2}\,|\varphi|^2\,dx+q\,\lambda_1(\Omega;q)\,\left(\int_\Omega U^{q-1}\,\varphi\,dx\right)^2\ge 0.
\end{equation}
We now take $\mathcal{F}\subset \mathcal{D}^{1,2}_0(\Omega)$ a vector subspace with dimension $m\ge 2$, we claim that 
\begin{equation}
\label{maduro}
\mbox{there exists $\varphi_{\mathcal{F}}\in \mathcal{F}\setminus\{0\}$ such that } \int_\Omega U^{q-1}\,\varphi_{\mathcal{F}}\,dx=0.
\end{equation}
Indeed, let us take two linearly independent functions $\varphi_1,\varphi_2\in\mathcal{F}$. If one of these two functions has property \eqref{maduro} we are done. Otherwise, it results
\[
\int_\Omega U^{q-1}\,\varphi_1\,dx=\alpha\not =0\qquad \mbox{ and }\qquad \int_\Omega U^{q-1}\,\varphi_2\,dx=\beta\not =0.
\]
By defining $\varphi=\beta\,\varphi_1-\alpha\,\varphi_2$, we would get that $\varphi\in\mathcal{F}\setminus\{0\}$ has property \eqref{maduro}.
We can exploit this fact and the Courant-Fischer-Weyl min-max principle, to get
\[
\begin{split}
\mu_2&=\min_{\mathcal{F}}\left\{\max_{\varphi\in\mathcal{F}} \frac{\displaystyle\int_\Omega |\nabla \varphi|^2\,dx-(q-1)\,\lambda_1(\Omega;q)\,\int_\Omega U^{q-2}\,|\varphi|^2\,dx}{\displaystyle\int_\Omega |\varphi|^2\,dx}\right\}\\
&\ge \min_{\mathcal{F}}\frac{\displaystyle\int_\Omega |\nabla \varphi_{\mathcal{F}}|^2\,dx-(q-1)\,\lambda_1(\Omega;q)\,\int_\Omega U^{q-2}\,|\varphi_{\mathcal{F}}|^2\,dx}{\displaystyle\int_\Omega |\varphi_{\mathcal{F}}|^2\,dx}\ge 0,
\end{split}
\] 
thanks to \eqref{egolo} and \eqref{maduro}.
\end{proof}

\begin{oss}
The previous result can also be rephrased by saying that for $2<q<2^*$, a first $q-$eigenfunction has always {\it Morse index} equal to $1$, see for example \cite[Section 2]{BL}.
\end{oss}

\section{Critical set of a first $q-$eigenfunction}

The following simple result is useful in order to give a rough estimate on the critical set of a first $q-$eigenfunction. The result should be quite well-known, but we have not been able to trace it back in the literature. We thus give a proof.
\begin{lm}
\label{lm:criticalset}
Let $2<q<2^*$ and let $\Omega\subset\mathbb{R}^N$ be a $q-$admissible open set. Let $U\in\mathcal{D}^{1,2}_0(\Omega)$ be a positive first $q-$eigenfunction.
Then for every $-1<\alpha<0$ we have
\[
|U_{x_j}|^\frac{\alpha+2}{2}\in W^{1,2}_{\rm loc}(\Omega),\qquad j=1,\dots,N.
\]
Moreover, for every $0<\beta<1/2$ we also have
\[
\frac{1}{|\nabla U|^\beta}\in L^1_{\rm loc}(\Omega).
\]
\end{lm}
\begin{proof}
We know that $U$ weakly solves
\[
-\Delta U=\lambda_1(\Omega;q)\,U^{q-1},\qquad \mbox{ in }\Omega.
\]
Since we have $U\in L^\infty(\Omega)$ by Proposition \ref{prop:Linfty}, the right-hand side is in particular in $L^2(\Omega)$. Thus, we get $U\in H^2_{\rm loc}(\Omega)$ by the classical Nirenberg's method of incremental quotients. By using a test function of the form $\varphi_{x_j}$, with $\varphi\in C^\infty_0(\Omega)$, and then integrating by parts, we can obtain 
\[
\int_\Omega \langle\nabla U_{x_j},\nabla \varphi\rangle\,dx=(q-1)\,\lambda_1(\Omega;q)\,\int_\Omega U^{q-2}\,U_{x_j}\,\varphi\,dx.
\]
By density, the same equation still holds if $\varphi\in W^{1,2}(\Omega)$, with compact support contained in $\Omega$. In particular, we can take\footnote{Observe that the function $f_\varepsilon(t)=t\,(\varepsilon+t^2)^\frac{\alpha}{2}$ is $C^1$ and has bounded derivative. Thus $f_\varepsilon(U_{x_j})\in W^{1,2}_{\rm loc}(\Omega)$ and the test function is admissible.}
\[
\varphi=\eta^2\,(\varepsilon+|U_{x_j}|^2)^\frac{\alpha}{2}\,U_{x_j},
\]
where $\varepsilon>0$ and $-1<\alpha<0$. Here $\eta\in C^\infty_0(B_R)$ is a standard nonnegative cut-off function, with $B_R\Subset \Omega$ and $\eta\equiv 1$ on $B_r\subset B_R$.
We thus obtain
\[
\begin{split}
\int_\Omega |\nabla U_{x_j}|^2\,(\varepsilon+|U_{x_j}|^2)^\frac{\alpha}{2}\,\eta^2\,dx&+\alpha\,\int_\Omega |\nabla U_{x_j}|^2\,(\varepsilon+|U_{x_j}|^2)^\frac{\alpha-2}{2}\,|U_{x_j}|^2\,\eta^2\,dx\\
&\le2\, \int_\Omega |\nabla U_{x_j}|\,|\nabla \eta|\,\eta\,|U_{x_j}|\,(\varepsilon+|U_{x_j}|^2)^\frac{\alpha}{2}\,dx\\
&+(q-1)\,\lambda_1(\Omega;q)\,\int_\Omega U^{q-2}\,|U_{x_j}|^2\,(\varepsilon+|U_{x_j}|^2)^\frac{\alpha}{2}\,\eta^2\,dx.
\end{split}
\]
By recalling that $\alpha<0$, we get
\[
\alpha\,\int_\Omega |\nabla U_{x_j}|^2\,(\varepsilon+|U_{x_j}|^2)^\frac{\alpha-2}{2}\,|U_{x_j}|^2\,\eta^2\,dx\ge \alpha\,\int_\Omega |\nabla U_{x_j}|^2\,(\varepsilon+|U_{x_j}|^2)^\frac{\alpha}{2}\,\eta^2\,dx.
\]
By further using Young's inequality, the $L^\infty$ estimate on $U$ and the properties of $\eta$, we get
\[
\begin{split}
(1+\alpha)\,\int_\Omega |\nabla U_{x_j}|^2\,(\varepsilon+|U_{x_j}|^2)^\frac{\alpha}{2}\,\eta^2\,dx&\le \frac{1}{\tau}\,\int_\Omega |\nabla \eta|^2\,|U_{x_j}|^2\,(\varepsilon+|U_{x_j}|^2)^\frac{\alpha}{2}\,dx\\
&+ \tau\,\int_\Omega |\nabla U_{x_j}|^2\,(\varepsilon+|U_{x_j}|^2)^\frac{\alpha}{2}\,\eta^2\,dx\\
&+(q-1)\,\lambda_1(\Omega;q)\,\|U\|^{q-2}_{L^\infty(\Omega)}\, \int_{B_R} (\varepsilon+|U_{x_j}|^2)^\frac{\alpha+2}{2}\,dx.
\end{split}
\]
By observing that $1+\alpha>0$, we can take $\tau=(1+\alpha)/2$ and absorb the term with the Hessian of $U$ on the right-hand side. This gives
\[
\begin{split}
\int_{B_r} |\nabla U_{x_j}|^2\,(\varepsilon+|U_{x_j}|^2)^\frac{\alpha}{2}\,dx&\le \left(\frac{2}{1+\alpha}\right)^2\,\frac{1}{(R-r)^2}\,\int_{B_R} (\varepsilon+|U_{x_j}|^2)^\frac{\alpha}{2}\,dx\\
&+\frac{2\,(q-1)}{1+\alpha}\,\lambda_1(\Omega;q)\,\|U\|^{q-2}_{L^\infty(\Omega)}\, \int_{B_R} (\varepsilon+|U_{x_j}|^2)^\frac{\alpha+2}{2}\,dx.
\end{split}
\]
By introducing the function 
\[
F_\varepsilon(t)=\int_0^t (\varepsilon+\tau^2)^\frac{\alpha}{4}\,d\tau,
\]
from the previous estimate we have that for $0<\varepsilon<1$
\[
\int_{B_r} |F_\varepsilon(U_{x_j})|^2\,dx+\int_{B_r} |\nabla F_\varepsilon(U_{x_j})|^2\,dx\le C_\alpha,
\]
with $C>0$ independent of $\varepsilon$. This shows that $F_\varepsilon(U_{x_j})$ converges to $F_0(U_{x_j})=|U_{x_j}|^{(\alpha+2)/2}$ weakly in $W^{1,2}(B_r)$, as $\varepsilon$ goes to $0$. Thus in particular, we have
\[
|U_{x_j}|^\frac{\alpha+2}{2}\in W^{1,2}(B_r).
\]
Thanks to the arbitrariness of the ball $B_r$, we get the desired property of $U_{x_j}$.
\par
In order to prove that a negative power of $|\nabla U|$ is locally summable, we first observe that from the previous property, we also get
\begin{equation}
\label{sobolevity}
|\nabla U|^\frac{\alpha+2}{2}\in W^{1,2}_{\rm loc}(\Omega).
\end{equation}
Then we test the equation for $U$ with
\[
\varphi=\eta\,\frac{1}{(\varepsilon+|\nabla U|)^\beta},
\]
where $\eta$ is as before and $0<\beta<1/2$. We get
\[
\begin{split}
\lambda_1(\Omega;q)\,\int_{B_r} \frac{U^{q-1}}{(\varepsilon+|\nabla U|)^\beta}\,dx&\le \int_{B_R} |\nabla \eta|\,\frac{|\nabla U|}{(\varepsilon+|\nabla U|)^\beta}\,dx \\
&+\beta\,\int_{B_R} \frac{|\nabla U|}{(\varepsilon+|\nabla U|)^{\beta+1}}\,|\nabla |\nabla U||\,dx\\
&\le \frac{C}{R-r}\,\int_{B_R} |\nabla U|^{1-\beta}+\beta\,\int_{B_R} |\nabla U|^{-\beta}\,|\nabla |\nabla U||\,dx\\
&=\frac{1}{R-r}\,\int_{B_R} |\nabla U|^{1-\beta}+\frac{\beta}{1-\beta}\,\int_{B_R} \Big|\nabla |\nabla U|^{1-\beta}\Big|\,dx.
\end{split}
\]
We now observe that since $\beta<1/2$, then $1-\beta>1/2$ and thus the last integral is finite, thanks to \eqref{sobolevity} with $\alpha=-2\,\beta$. By taking the limit as $\varepsilon$ goes to $0$, this shows that
\[
\int_{B_r} \frac{U^{q-1}}{|\nabla U|^\beta}\,dx<+\infty.
\]
The claimed integrability of $|\nabla U|^{-\beta}$ now follows by observing that $U\ge c>0$ on $B_r$, by the minimum principle.
\end{proof}

\begin{oss}
The previous result permits to infer that for every $K\Subset \Omega$, the critical set $\{x\in K\, :\, |\nabla U(x)|=0\}$ has $N-$dimensional equal to $0$. This is quite a poor information, which is however enough in order to accomplish {\bf Step 1} in Theorem \ref{teo:simpleball} above.
\par
There is a vast literature on the problem of estimating the critical set for solutions of linear elliptic PDEs of the form
\[
\mathrm{div}(A(x)\,\nabla u)+\langle \mathbf{b}(x),\nabla u\rangle=0,
\]
see for example the by now classical reference \cite{HHON}.
We point out that a first positive $q-$eigenfunction $U$ can be regarded as a solution of the linear equation
\[
-\Delta U=c\,U, \quad \mbox{ in }\Omega,\qquad \mbox{ where } c(x)=\lambda_1(\Omega;q)\,U^{q-2}(x).
\]
However, this observation does not seem very useful, since well-known counter-examples show that for these equations an estimate of the critical set is not possible, see \cite[page 133]{Ma}.
\end{oss}

\end{document}